\documentclass[english]{amsart}
\usepackage[all]{xy}
\usepackage{varioref,babel}
\usepackage{enumerate,a4,amsfonts,amsmath,amsfonts,amssymb,amscd}
\usepackage{amsxtra,eucal}
\usepackage[latin1]{inputenc}
\usepackage{tabularx}

\usepackage{mathrsfs}
\usepackage{hyperref}
\usepackage{mathtools}
\usepackage{color}

\usepackage{dsfont} 

\usepackage[foot]{amsaddr}

\newcommand{\remove}[1]{} 

\newcommand{\N}{\mathbb{N}}
\newcommand{\Q}{\mathbb{Q}}
\newcommand{\R}{\mathbb{R}}
\newcommand{\Z}{\mathbb{Z}}

\newcommand{\calB}{{\mathcal{B}}}

\newcommand{\calO}{{\mathcal{O}}}

\newcommand{\calR}{{\mathcal{R}}}

\newcommand{\calT}{{\mathcal{T}}}

\newcommand{\catC}{{\mathscr{C}}}

\newcommand{\catG}{{\mathscr{G}}}

\newcommand{\catP}{{\mathscr{P}}}

\newcommand{\catU}{{\mathscr{U}}}

\newcommand{\fraka}{{\mathfrak{a}}}

\newcommand{\frakm}{{\mathfrak{m}}}

\newcommand{\frakp}{{\mathfrak{p}}}

\newcommand{\bfG}{{\mathbf{G}}}
\newcommand{\bfH}{{\mathbf{H}}}

\newcommand{\smallSMatII}[4]{\left[\begin{smallmatrix} {#1} & {#2} \\ {#3} &
{#4} \end{smallmatrix}\right]}

\newcommand{\invlim}{\underleftarrow{\lim}\,}

\newcommand{\suchthat}{\,:\,}

\newcommand{\quo}[1]{\overline{#1}}

\newcommand{\veps}{\varepsilon}

\newcommand{\Trings}[1]{\left< #1 \right>}

\DeclareMathOperator{\Cent}{Cent}

\DeclareMathOperator{\coker}{coker}
\DeclareMathOperator{\corad}{corad} %
\DeclareMathOperator{\End}{End} %
\DeclareMathOperator{\Fix}{Fix} %
\DeclareMathOperator{\Gal}{Gal} %

\DeclareMathOperator{\Hom}{Hom} %
\DeclareMathOperator{\id}{id}
\DeclareMathOperator{\im}{im} %

\DeclareMathOperator{\Jac}{Jac} %
\DeclareMathOperator{\Mor}{Mor} %
\newcommand{\op}{\mathrm{op}}

\DeclareMathOperator{\Spec}{Spec}
\DeclareMathOperator{\Sym}{Sym}

\newtheorem{thm}{Theorem}[section]
\newtheorem*{thm*}{Theorem}
\newtheorem{lem}[thm]{Lemma}
\newtheorem{prp}[thm]{Proposition}
\newtheorem{cor}[thm]{Corollary}

\newtheorem{cnj}[thm]{Conjecture}

\newtheorem{que}[thm]{Question}

\newtheorem{baseexample}[thm]{Example} 
\newtheorem{baseremark}[thm]{Remark} 

\newenvironment{example}
{\begin{baseexample}\rm}{\end{baseexample}}
\newenvironment{remark}
{\begin{baseremark}\rm}{\end{baseremark}}

\newcommand{\units}[1]{{#1^\times}}   

\newcommand{\Herm}[2][]{\mathscr{U\!\!H}^{#1}(#2)}     
\newcommand{\Hyp}[2][]{\mathds{h}^{#1}_{#2}}      
\newcommand{\herm}[2][]{\mathscr{H}^{#1}({#2})}

\newcommand{\rMod}[1]{{\mathrm{Mod}\textrm{-}{#1}}} 

\newcommand{\rproj}[1]{\catP(#1)}                   

\newcommand{\lAd}[1]{{#1}_\ell}   

\newcommand{\nMat}[2]{\mathrm{M}_{#2}(#1)}

\newcommand{\uU}{{\mathbf{U}}}
\newcommand{\uGL}{{\mathbf{GL}}}
\newcommand{\fppf}{\mathrm{fppf}}
\newcommand{\et}{\mathrm{\acute{e}t}}
\newcommand{\sh}{{\mathrm{sh}}}

\newcommand{\rmH}{\mathrm{H}}

\numberwithin{equation}{section}

\title[Rationally Isomorphic Hermitian Forms]{Rationally Isomorphic Hermitian Forms and Torsors of Some Non-Reductive Groups}

\author{Eva Bayer-Fluckiger$^1$}
\author{Uriya A.\ First$^2$}
\address{$^1$\'{E}cole Polytechnique F\'{e}d\'{e}rale de Lausanne, Switzerland.}
\address{$^2$University of British Columbia, Canada.}

\date{\today}

\thanks{
The first named author is partially supported  by  an SNFS grant \#200021-163188.
The second named author has performed the research at EPFL, the Hebrew University of Jerusalem and the University
of British Columbia (in this order), where he was supported by an SNFS grant \#IZK0Z2\_151061,
an ERC grant \#226135, and the UBC Mathematics Department, respectively.
}

\begin{document}

\maketitle

\begin{abstract}
    Let $R$ be a semilocal Dedekind domain.
    Under certain assumptions, we show that two (not necessarily unimodular) hermitian forms over an
    $R$-algebra with involution, which  are \emph{rationally ismorphic} and have isomorphic semisimple \emph{coradicals},
    are in fact isomorphic.
    The same result is also obtained for quadratic forms equipped with an  action of a finite group.
    The results have cohomological restatements that resemble the Grothendieck--Serre conjecture, except
    the group schemes involved are not reductive.
    We show that these group schemes are closely related to group schemes arising in Bruhat--Tits theory.
\end{abstract}

\setcounter{section}{-1}

\section{Introduction}
\label{section:intro}

    Let $R$ be a  discrete valuation ring with $2\in\units{R}$, and let $F$ be its fraction field.
    The following theorem is well-known (see for instance  \cite[Th.~1]{Fi14B} for a short proof):

    \begin{thm}\label{TH:classical-result}
        Let $f,f'$ be two unimodular quadratic forms over $R$. If $f$ and $f'$ become isomorphic
        over $F$, then they are isomorphic over $R$.
    \end{thm}

    Over the years, this result has been generalized in many ways; see for instance \cite{Coll79}
    and \cite{Panin05} for surveys. Many of the generalizations are consequences of  the following conjecture:

    \begin{cnj}[Grothendieck \cite{Groth58}, Serre \cite{Serre58}]
        Let $R$ be a regular local integral  domain with fraction field $F$.
        Then for every  reductive  group scheme $\bfG$ over $R$, the induced map
        \[
        \rmH^1_{\et}(R,\bfG)\to \rmH^1_{\et}(F,\bfG)
        \]
        is  injective.
    \end{cnj}

    The conjecture can also be made for non-connected group schemes whose neutral
    component is reductive (although it is not true in this generality \cite[p.~18]{Coll79});
    a widely studied case is the orthogonal group and its forms.

    To see the connection to Theorem~\ref{TH:classical-result}, fix a unimodular quadratic
    space $(P,f)$ and let $\uU(f)$ denote the group scheme of isometeries of $f$ (the isometries
    of $f$ are the $R$-points of $\uU(f)$, denoted $U(f)$).
    Then  isomorphism classes of unimodular quadratic forms on the $R$-module $P$ correspond to $\rmH^1_{\text{\'{e}t}}(R,\uU(f))$
    (see for instance \cite[Ch.~III]{Kn91}).
    Thus, verifying the conjecture for $\uU(f)$ implies Theorem~\ref{TH:classical-result}.
    In this special case, the conjecture was proved when $\dim R\leq 2$ (\cite[Cor.~2]{Ojan82})
    or $R$ contains a field (\cite[Th.~9.2]{OjanPanin01}).

    The general Grothendieck--Serre conjecture  was recently proved by Fedorov and Panin in case $R$ contains
    a field $k$; see   \cite{FedPanin13} for the case where $k$ is infinite
    and \cite{Panin14}
    for the case where $k$ is finite.
    Many special cases were known before; see \cite{FedPanin13} and the references therein.
    In particular, Nisnevich \cite{Nisne84} proved the conjecture when $\dim R=1$.

    \medskip

    Recently, Theorem~\ref{TH:classical-result} was extended in a different direction by Auel, Parimala and
    Suresh \cite{AueParSur14}. Let $R$ denote a semilocal Dedekind domain with $2\in\units{R}$ henceforth.
    A quadratic form $f$ over $R$ has \emph{simple degeneration of multiplicity $1$} if its determinant is square free in $R$.
    They show:

     \begin{thm}[{\cite[Cor.~3.8]{AueParSur14}}]\label{TH:APS}
        Let $f,f'$ be
        two quadratic forms over $R$ having simple degeneration of multiplicity one.  If $f$ and $f'$ are isomorphic
        over $F$, then they are isomorphic over $R$.
    \end{thm}


	Note that the forms $f,f'$ in the theorem may be non-unimodular. When this is the case, they  can still be viewed
    as elements of $\rmH^1_{\text{\'{e}t}}(R,\uU(f))$, but $\uU(f)$   no longer has a reductive
    neutral component,
    so the theorem does not follow from the Grothendieck--Serre conjecture.

\medskip

    This is the starting point of our paper.  Our aim is to put Theorem~\ref{TH:APS} in a
    different perspective, and to study how far one can generalize it. Our point of view
    is inspired by the treatment of non-unimodular forms in \cite{BayerFain96}, \cite{BayerMold12} and \cite{BayFiMol13}.
    Roughly speaking, these
    works reduce the treatment of (systems of) non-unimodular forms to (single) unimodular
    forms over a different base ring.

 \medskip
 Let us start by defining the notion of a nearly unimodular hermitian form, a notion that extends the one
 considered by Auel, Parimala and Suresh. For any ring with involution $(A,\sigma)$, we say
 that a hermitian form $f:P\times P\to A$ is {\it nearly unimodular} if the cokernel of the homomorphism $P \to P^*$ induced by $f$
 is a semisimple $A$-module. We denote this cokernel by $\corad(f)$ and call it the  \emph{coradical} of $f$.

 Note that   a quadratic form over $R$ having simple degeneration
 of multiplicity $1$ is nearly unimodular.
     The main result of the paper is the following generalization of Theorem~\ref{TH:APS}:




    \begin{thm*}[cf.\ Th.~\ref{TH:non-unimodular-forms}]
        Let $A$ be a \emph{hereditary} $R$-order, and let
        $\sigma:A\to A$ be an $R$-involution.
        \begin{enumerate}[(i)]
        \item Let $f,f':P\times P\to A$ be  two nearly unimodular hermitian forms over $(A,\sigma)$
        whose
        coradicals are isomorphic
    	as $A$-modules.
 		Then $f_F\cong f'_F$ implies $f\cong f'$.
        \item Any  unimodular hermitian form over $(A_F,\sigma_F)$
        is obtained by base change from a nearly unimodular hermitian form  over $(A,\sigma)$.
        \end{enumerate}
    \end{thm*}

    Recall that an $R$-order is an $R$-algebra $A$ which is $R$-torsion-free and finitely generated  as an $R$-module.
    The $R$-order $A$ is hereditary if its one-sided ideals are projective. Notable examples of hereditary orders
    include \emph{maximal orders}.

\medskip

    Let $f$ be as in part (i) of the theorem and assume further that $f$ is {unimodular}.
    Then $\uU(f)$ is a smooth affine group scheme over $ R$,
    and part (i) of the theorem can be restated as:

    \begin{thm*}[cf.\ Th.~\ref{TH:main-coh}]
        The map $\rmH^1_\et(R,\uU(f))\to \rmH^1_\et(F,\uU(f))$ is injective.
    \end{thm*}

    Note that
    while this resembles the  Grothendieck--Serre conjecture, the
    neutral component of $\uU(f)$ is not always reductive
    (Example~\ref{EX:reductive}).

    It turns out that the
    group schemes $\uU(f)$ can be given an alternative description using \emph{Bruhat-Tits theory}
    (Corollary~\ref{CR:parahoric}).
    This  description actually gives rise to a wider family of non-reductive group schemes
    over $R$, suggesting that the Grothendieck--Serre conjecture (in the case $\dim R=1$)
    might extend to these group schemes (Question~\ref{QU:GS}).

\medskip

    We note that Theorem~\ref{TH:non-unimodular-forms}(i) fails for arbitrary
    non-unimodular hermitian forms (Remark~\ref{RM:non-semisimple-corad}), or if $A$ is assumed to be a general $R$-order
    (Remark~\ref{RM:arbitrary-orders}).

\medskip

    As an application of Theorem~\ref{TH:non-unimodular-forms}(i), we prove a result about quadratic forms
    equipped with an action of a finite group $\Gamma$. 
    Recall that a $\Gamma$-form (over $R$) is a pair $(P,f)$, where $P$ is a
    finitely generated right $R\Gamma$-module, and $f : P \times P\to R$
    is a symmetric
    $R$-bilinear form such that $f(xg,yg) = f(x,y)$ for all $x,y \in P$ and  $g \in \Gamma$. We say that a $\Gamma$-form is {\it nearly unimodular}
    if it is nearly unimodular as a bilinear form over $R$. We prove:

    \begin{thm*}[cf.\ Th.~\ref{TH:Gamma-forms}]
        Let $(P,f)$ and $(P',f')$ be two nearly unimodular $\Gamma$-forms over $R$.
        Assume that $|\Gamma|\in\units{R}$, and that the coradicals of $f$ and $f'$ are isomorphic
        $R\Gamma$-modules.
        Then $(P_F,f_F)\cong (P'_F,f'_F)$ as $\Gamma$-forms implies
        $(P,f)\cong (P',f')$ as $\Gamma$-forms. Furthermore, any unimodular $\Gamma$-form
        over $F$ can be obtained by base change from a nearly unimodular $\Gamma$-form over $R$.
    \end{thm*}

    The cohomological results of this paper were written with the help
    of Mathieu Huruguen, and we thank him for his contribution.
    We also thank Jean-Pierre Serre and the anonymous referee for many beneficial comments and suggestions.

\medskip

    The paper is organized as follows: Sections~\ref{section:preliminaries} and~\ref{section:orders}
    recall hermitian forms and hereditary orders, respectively.
    Section~\ref{section:hermitian} is the technical heart of the paper, and it contains
    the proof of Theorem~\ref{TH:non-unimodular-forms}(i) in the unimodular case (Theorem~\ref{TH:forms-over-hereditary-has-genus-one});
    the proof uses patching results from \cite{BayFi14}.
    Following is Section~\ref{section:non-unimodular}, which proves Theorem~\ref{TH:non-unimodular-forms},
    deriving  part (i) from the unimodular case using results of \cite{BayerFain96} and \cite{BayFiMol13}.
    Theorem~\ref{TH:main-coh} is the subject matter of Section~\ref{section:cohomological}.
    In Section~\ref{section:BT}, we relate the group schemes appearing in Theorem~\ref{TH:main-coh}
    with group schemes arising in Bruhat--Tits theory  and, based on that,
    suggest an extension of the Grothendieck--Serre conjecture.
    Section~\ref{section:Gamma-forms} contains the aforementioned application to
    $\Gamma$-forms (Theorem~\ref{TH:Gamma-forms}).

\section{Hermitian Forms}
\label{section:preliminaries}

    We start by recalling hermitian forms over rings.
    We refer the reader to \cite{Kn91} and \cite{SchQuadraticAndHermitianForms} for details and proofs.

\subsection{Hermitian Forms}
\label{subsection:hermitian-forms}

    Let $(A,\sigma)$ be a ring with involution and let $u\in \Cent(A)$ be an element satisfying $u^\sigma u=1$.
    Denote by $\rproj{A}$ the category of finitely generated projective right $A$-modules.
    A \emph{$u$-hermitian space} over $(A,\sigma)$ is a pair $(P,f)$
    such that $P\in\rproj{A}$ and $f:P\times P\to A$
    is a biadditive map satisfying
    \[
    f(xa,yb)=a^\sigma f(x,y)b\qquad\text{and}\qquad f(x,y)=f(y,x)^\sigma u
    \]
    for all $x,y\in P$ and $a,b\in A$. In this case, $f$ is called a \emph{$u$-hermitian form}
    on $P$.

    An \emph{isometry} from $(P,f)$ to another $u$-hermitian space $(P',f')$
    is a map $\phi:P\to P'$ such that $\phi$ is an isomorphism of $A$-modules
    and $f'(\phi x,\phi y)=f(x,y)$ for all $x,y\in P$. The group
    of isometries of $(P,f)$ is denoted $U(f)$.

    The \emph{orthogonal sum} of two hermitian spaces is defined in the obvious way
    and is denoted using the symbol ``$\oplus$''.

\medskip

    For every $P\in\rproj{A}$, define $P^*=\Hom_A(P,A)$.
    We view $P^*$ as a \emph{right} $A$-module by
    setting $(\phi a)x=a^\sigma (\phi x)$ for all $\phi\in P^*$, $a\in A$, $x\in P$.
    The assignment $P\mapsto P^*:\rproj{A}\to\rproj{A}$ is a contravariant functor,
    a duality in fact. Indeed, the map $\omega_P:P\to P^{**}$ given by
    $(\omega_Px)\phi =(\phi x)^\sigma u$ is well-known to be a {natural isomorphism}.
    Every $u$-hermitian space $(P,f)$ induces a map
    \[
    f_\ell : P\to P^*
    \]
    given by $(f_\ell x)(y)=f(x,y)$ for all $x,y\in P$.
    We say that $f$ is \emph{unimodular} if $f_\ell$ is bijective.
    We denote by $\Herm[u]{A,\sigma}$ (resp.\ $\herm[u]{A,\sigma}$)
    the category of unimodular (resp.\ arbitrary) $u$-hermitian spaces
    over $(A,\sigma)$ with isometries as morphisms.

\medskip

    Let $P\in\rproj{A}$. The \emph{hyperbolic $u$-hermitian space}
    associated with $P$ is \linebreak $(P\oplus P^*,\Hyp{P})$, where $\Hyp{P}(x\oplus \phi,x'\oplus \phi')=\phi x' +(\phi'x)^\sigma u$
    for all $x,x'\in P$ and $\phi,\phi'\in P^*$.
    In case $A=B\times B^\op$ and $\sigma$ is the \emph{exchange involution} $(a,b^\op)\mapsto (b,a^\op)$,
    every hermitian space $(P,f)\in \Herm[u]{A,\sigma}$ is isomorphic to $(Q\oplus Q^*,\Hyp{Q})$
    for $Q=P(1_B,0_B^\op)$. In particular, $(P,f)$ is determined up to isometry by $P$.

\medskip

    Let $R$ be  a commutative ring and let $S$ be a commutative
    $R$-algebra. Assume henceforth that $(A,\sigma)$
    is an $R$-algebra with an $R$-involution. We let $A_S=A\otimes_R S$ and $\sigma_S=\sigma\otimes_R\id_S$.
    In addition, for every $P\in\rproj{A}$, we set $P_S=P\otimes_RS\in\rproj{A_S}$,
    where $P_S$ is viewed as a right $A_S$-module by linearly
    extending $(x\otimes s)(a\otimes s')=xa\otimes ss'$
    for all $x\in P$, $a\in A$, $s,s'\in S$.

    Every  $u$-hermitian space $(P,f)\in\herm[u]{A,\sigma}$ gives
    rise to a $u$-hermitian space $(P_S,f_S)\in\herm[u]{A_S,\sigma_S}$ with $f_S$ is given by
    \[
    f_S(x\otimes s,x'\otimes s')=f(x,x')\otimes ss'\qquad\forall\, x,x'\in P,\, s,s'\in S\ .
    \]
    It is well-known that if $(P,f)$ is unimodular, then so is $(P_S,f_S)$.

	When $A\in\rproj{R}$ and $2\in\units{R}$,
    the assignment $S\mapsto U(f_S)$ is the functor of points
    of an affine group scheme over $R$, denoted $\uU(f)$.
    This group scheme is smooth when $f$ is unimodular;
    see~\cite[Apx.]{BayFi14}.
    We further let $\uU(A,\sigma)$ denote the affine
    group scheme over $ R$ whose $S$-points are given by $\uU(A,\sigma)(S)=U(A_S,\sigma_S):=\{a\in A_S\suchthat a^\sigma a=1\}$.

\medskip

    We shall need the following well-known strengthening of Witt's Cancellation Theorem. A proof
    can be found in \cite[Th.~7.9.1]{SchQuadraticAndHermitianForms}, for instance.

    \begin{thm}\label{TH:witt-cancellation}
        Let $F$ be a field of characteristic not $2$.
        Assume $A$ is a finite dimensional $F$-algebra and $\sigma$ is $F$-linear.
        Then cancellation holds for unimodular $u$-hermitian forms over $(A,\sigma)$.
    \end{thm}


\subsection{Transfer into the Endomorphism Ring}
\label{subsection:transfer}

    We now recall the method of transfer into the endomorphism ring.
    This is in fact a special case of transfer in \emph{hermitian categories};
    see \cite[Pr.~2.4]{QuSchSch79} or \cite[II.\S3]{Kn91}.

\medskip

    Let $(E,\tau)$ be a ring with involution. Two elements
    $a,b\in E$ are said to be $\tau$-congruent, denoted $a\sim_\tau b$,
    if there exists
    $v\in\units{E}$ such that $a=v^\tau bv$. This is  an equivalence
    relation.
    Let
    \[
    \units{\Sym}(E,\tau)=\{a\in \units{E}\suchthat a^\tau=a\}\qquad\text{and}\qquad
    \mathrm{H}(E,\tau)=\units{\Sym}(E,\tau)/\sim_\tau\ .
    \]
    The following well-known result allows one to translate statements
    about isometry of hermitian forms into statements about $\tau$-congruence.

    \begin{prp}\label{PR:transfer-into-edomorphism-ring}
        Let $(A,\sigma)$ be a ring with involution, and let $u\in \Cent(A)$
        be an element satisfying $u^\sigma u=1$.
        Let $(P,f)$ be a unimodular $u$-hermitian space over $(A,\sigma)$,
        and let $\Herm[u]{P}$ denote the
        set of unimodular
        $u$-hermitian spaces over $(A,\sigma)$ with base module $P$.
        Let $E=\End_A(P)$, and define $\tau:E\to E$  by
        $g^\tau=f_\ell^{-1}g^*f_\ell$. Equivalently, $g^\tau$ is determined
        by the identity $f(g^\tau x,y)=f(x, gy)$. Then $(E,\tau)$ is a ring with involution and
        there is a one-to-one correspondence
        \[
        \Herm[u]{P}/\cong \quad\longleftrightarrow\quad \mathrm{H}(E,\tau)
        \]
        given by sending the isometry class of $h\in\Herm[u]{P}$
        to the $\tau$-congruence class of $f_\ell^{-1}h_\ell\in E$.
    \end{prp}

    \begin{proof}
        See for instance \cite[Lm.~3.8.1]{BayParSerr13}.
    \end{proof}

    \begin{remark}\label{RM:transfer-is-OK-with-scalar-ext}
        The correspondence in Proposition~\ref{PR:transfer-into-edomorphism-ring} is compatible with scalar extension:
        Let $S$ be a commutative $R$-algebra and suppose $(A,\sigma)$ is an involutary
        $R$-algebra. Then there
        is a natural isomorphism $\End_{A_S}(P_S)\cong E_S$ (see for instance \cite[Lm.~1.2]{BayFi14})
        and the diagram
        \[
        \xymatrix{
        \Herm[u]{P}/\cong \ar@{<->}[r] \ar[d]_{h\mapsto h_S} &  \mathrm{H}(E,\tau) \ar[d]^{a\mapsto a\otimes 1} \\
        \Herm[u]{P_S}/\cong \ar@{<->}[r]  &  \mathrm{H}(E_S,\tau_S)
        }
        \]
        commutes. Moreover, the isomorphism $\End_{A_S}(P_S)\cong E_S$ restricts
        to a natural isomorphism $U(f_S)\cong U(E_S,\tau_S)$, and hence
        $\uU(f)\cong \uU(E,\tau)$.
    \end{remark}

    \begin{remark}\label{RM:summand-involution}
        In Proposition~\ref{PR:transfer-into-edomorphism-ring},
        if $(P,f)=(P_1,f_1)\oplus(P_2,f_2)$ and $e\in E$ is the orthogonal projection
        of $P$ onto $P_1$, then $e^\tau=e$. Indeed, $f(ex,y)=f(x,ey)$ for all $x,y\in P$.
    \end{remark}

\section{Hereditary Orders}
\label{section:orders}

    This section recalls facts about hereditary orders that will be used in the sequel.
    Unless specified otherwise, $R$ is a Dedekind domain with fraction field
    $F$.
    For every $0\neq\frakp\in \Spec(R)$, denote by $R_\frakp$ the localization
    of $R$ at $\frakp$, and let $\hat{R}_\frakp$ denote the completion of $R_\frakp$. The
    Jacobson radical of a ring $A$ is denoted $\Jac(A)$.

\subsection{Generalities on Orders}
\label{subsection:orders}

    Let $E$ be a finite-dimensional $F$-algebra.
    Recall that an \emph{$R$-order} in $E$ is an $R$-subalgebra
    $A$ such that $A$ is finitely generated as an $R$-module and $A\cdot F=E$.
    Equivalently, an $R$-algebra $A$ is an $R$-order (in some $F$-algebra, necessarily
    isomorphic to $A_F:=A\otimes_R F$) if $A$ is $R$-torsion-free and finitely generated   as an $R$-module.
    Since $R$ is a Dedekind domain, this implies  $A\in\rproj{R}$ (\cite[\S2E]{La99}).

\medskip

    Let $A$ be an $R$-order. Recall that  $A$ is  \emph{hereditary}
    if all one-sided ideals of $A$ are projective,
    and $A$ is \emph{maximal} if $A$ is not
    properly contained in an $R$-order in $A_F$.
    See \cite{MaximalOrders} for  details and examples.

\medskip

	Recall further  that $E$ is a \emph{separable} $F$-algebra if $E$ is semisimple and
	$\Cent(E)$ is a product of separable field extension of $F$.
	
	There is a generalization of the notion of separability to $R$-algebras that will be
	needed  in Section~\ref{section:cohomological}:
	An $R$-algebra $A$ is   \emph{separable}  if $A$ is projective when viewed
    as a left $A\otimes_R A^\op$-module via $(a\otimes b^\op)x=axb$ ($a,b,x\in A$).
    This definition agrees with the definition in the previous paragraph when $R$ is a field (\cite[Cor.~II.2.4]{DeMeyIngr71SeparableAlgebras}).
   	Separable $R$-algebras with center $R$ are also called \emph{Azumaya}.
    The separable $R$-orders $A$ in $E$ can also be characterized
    as those which are \emph{unramified} in the  sense
    that for any $\frakp\in \Spec R$, the $k(\frakp)$-algebra $A\otimes_R k(\frakp)$ is  separable, where $k(\frakp)$
    is the fraction field of $R/\frakp$ (\cite[Cor.~II.1.7, Th.~II.7.1]{DeMeyIngr71SeparableAlgebras}).


    \begin{thm}[{\cite[Th.~1.7.1]{HijNish94}}]
    	A finite-dimensional $F$-algebra $E$ contains a hereditary $R$-order
    	if and only if $E$ is semisimple and the integral closure of $R$ in $\Cent(E)$,
    	denoted $Z$, is finitely
    	generated as an $R$-module. In this case
    	$E$ also has maximal $R$-orders, and $Z$ is contained in any hereditary $R$-order in $E$.
    \end{thm}

\remove{
    A finite-dimensional $F$-algebra $E$ contains a hereditary $R$-order
    if and only if $E$ is semisimple and the integral closure of $R$ in $\Cent(E)$,
    denoted $Z$, is finitely
    generated as an $R$-module (\cite[Th.~1.7.1]{HijNish94}). In this case
    $E$ also has maximal $R$-orders, and $Z$ is contained in any hereditary $R$-order in $E$.
}

   	The $R$-algebra $Z$  in the theorem is always finitely generated as an $R$-module when  $E$ is separable over $F$.
    Examples of simple $F$-algebras $E$ where this fails can occur,
    for example, when $R$ is not excellent.

    \begin{thm}\label{TH:separable-maximal-hereditary}
        Let $A$ be an $R$-order.
        If $A$ is separable, then $A$ is maximal,
        and if $A$ is maximal, then $A$ is hereditary.
    \end{thm}

    \begin{proof}
        The second statement follows from condition (H.$0$) in
        \cite[Th.~1.6]{HijNish94}, so we turn to the first
        statement.
        Assume $A$ is separable
        and let $S=\Cent(A)$. Then $A$ is Azumaya over $S$ and
        $S$ is separable over $R$ (\cite[Th.~3.8]{DeMeyIngr71SeparableAlgebras}).
        Since $R$ is integrally closed in $F$,
        and $S$ is separable and projective over $R$,
        the ring $S$ is integrally closed in $S_F=\Cent(A_F)$ (\cite[Th.~5.1]{Olivier83}).
        Let $A'$ be an $R$-order with $A\subseteq A'\subseteq A_F$,
        and
        let $S'$ be the centralizer of $A$ in $A'$. Then $S'\subseteq\Cent(A_F)$
        and $S'$ is integral over $R$, hence $S'=S$.
        In addition, since $A$ is Azumaya over $S$,
        the map $a\otimes s'\mapsto as':A\otimes_SS'\to A'$ is an isomorphism (\cite[Pr.~2.7]{Sa99}),
        so $A'=A$.
    \end{proof}

    \begin{thm}\label{TH:local-hereditary}
        Let $A$ be an $R$-order. Then $A$ is hereditary
        (resp.\ maximal) if and only if $A\otimes_R \hat{R}_\frakp$
        is hereditary (resp.\ maximal) for all $0\neq \frakp\in \Spec(R)$.
    \end{thm}

    \begin{proof}
        See \cite[Th.~6.6]{BayFi14} and \cite[Cor.~11.6]{MaximalOrders}.
    \end{proof}

    Let $A$ be an $R$-order in $E=A_F$, and let $M$ be a right $E$-module.
    Recall that a \emph{full $A$-lattice} in $M$ is a finitely generated
    $A$-submodule $L\subseteq M$ such that $LF=M$. Every right $A$-module $L$
    which is finitely generated and $R$-torsion-free is a full $A$-lattice in $L_F:=L\otimes_RF$.

    If $L$ and $L'$ are two full $A$-lattices in $M$ such that  $L\subseteq L'$,
    then $\mathrm{length}(L'/L)<\infty$ (see for instance \cite[Exer.~4.1]{MaximalOrders}).
    Furthermore, for all $A$-lattices $L$ and $L'$, we can embed $\Hom_A(L,L')$ in
    $\Hom_{A_F}(L_F,L'_F)$ via $\phi\mapsto \phi\otimes \id_F$.
    The image of this map is $\{\psi\in\Hom_{A_F}(L_F,L'_F)\suchthat\psi(L)\subseteq L'\}$.

    \begin{prp}\label{PR:lattices}
    	Let $A$ be a hereditary $R$-order, let $M$ be a right $A_F$-module,
    	and let $L$ be a full $A$-lattice in $M$. Let $L'=\Hom_A(L,A)$
    	and view it as a subset of $M':=\Hom_{A_F}(M,A_F)$. Then
    	$L\in\rproj{A}$ and $L=\{x\in M\suchthat \text{$\phi x\in A$ for all $\phi\in L'$}\}$.
		Furthermore, if  $M'$ is viewed as a left
		$A_F$-module via $(a\psi)x=\psi(xa)$ ($a\in A$, $\psi\in M'$, $x\in M$),  then $L'$
		is a full (left) $A$-lattice in $M'$.
    \end{prp}

    \begin{proof}
    	The module $L$ is finitely generated by definition.
    	By Kaplansky's Theorem \cite[Th.~2.24]{La99}, in order to prove that $L$ is
    	projective, it is enough to embed it in a free $A$-module. Since $A_F$ is semisimple,
    	$M$ embeds as a submodule of $A_F^n$ for some $n\in\N$. Viewing $L$ as a f.g.\
    	$A$-submodule of $A_F^n$, there is some
    	$0\neq a\in R$ such that $aL\subseteq A^n$, so $L$ is isomorphic to a sumodule of $A^n$.
    	
    	Now that $L$ is f.g.\ projective, we can choose a finite
    	\emph{dual basis}
    	for $L$  (see \cite[Lm.~2.9, Rm.~2.11]{La99}), namely,
    	there are $\{x_i\}_{i=1}^n\subseteq L$ and $\{\phi_i\}_{i=1}^n\subseteq \Hom_A(L,A)$
    	such that $\sum_ix_i\phi_ix=x$ for all $x\in L$. It is easy to see that $\{x_i,\phi_i\}_{i=1}^n$
    	is also a dual basis of $M$. Suppose that $x\in M$ satisfies
    	$\phi x\in A$ for all $\phi\in L'$. Then $x=\sum_i x_i\phi_ix\in L$,
    	proving $L\supseteq\{x\in M\suchthat \text{$\phi x\in A$ for all $\phi\in L'$}\}$.
    	The opposite inclusion is clear.
    	
    	Finally, note that for all $\psi\in L'$, we have $\psi=\sum_i(\psi x_i)\phi_i$.
    	Indeed, $\sum_i(\psi x_i)\phi_i y=\psi(\sum_ix_i\phi_iy)=\psi y$ for
    	all $y\in L$. This shows that $L'$ is finitely generated as left $A'$-module. Applying the same
    	argument with elements of $M'$ shows that $FL'=M'$, so $L'$ is a full (left) $A$-lattice
    	in $M'$.
    \end{proof}

\subsection{The Structure of Hereditary Orders}
\label{subsection:structure-of-hered-orders}

    We now recall the  structure theory of hereditary orders over complete discrete valuation rings.
    The general case can be reduced to this setting by Theorem~\ref{TH:local-hereditary}.
    Our exposition follows \cite[\S39]{MaximalOrders}; proofs and further details can be found
    there. The results recalled here are in fact true for any henselian DVR (\cite[p.\ 364]{MaximalOrders}).

    \emph{Throughout, $R$ is  assumed to be a complete DVR,
    and $\nu=\nu_F$ denotes the corresponding (additive) valuation on $F$.}

\medskip

    We first recall the structure of \emph{maximal} orders in division $F$-algebras.

    \begin{thm}\label{TH:structure-of-max-orders}
        Let $D$ be a finite dimensional division algebra over $F$. Then the valuation $\nu$ extends uniquely
        to a valuation $\nu_D$ on $D$. Furthermore, the ring $\calO_D:=\{a\in D\suchthat \nu_D(a)\geq 0\}$
        is an $R$-order, and it
        is the only maximal $R$-order in $D$.
    \end{thm}

    \begin{proof}
        See \cite[\S12]{MaximalOrders}.
    \end{proof}

    We denote the unique maximal right (and left) ideal of $\calO_D$ by $\frakm_D$.
    The quotient $k_D:=\calO_D/\frakm_D$ is a finite-dimensional
    division $R/\frakm$-algebra, which is not
    central in general.

\medskip

    Given a ring $A$ and ideals $(\fraka_{ij})_{i,j}$, we let
    \[
    \left[
    \begin{array}{ccc}
    (\fraka_{11}) & \dots & (\fraka_{1r}) \\
    \vdots & & \vdots \\
    (\fraka_{r1}) & \dots & (\fraka_{rr})
    \end{array}
    \right]^{(n_1,\dots,n_r)}
    \]
    denote the set of block matrices $(X_{ij})_{1\leq i,j\leq r}$
    for which $X_{ij}$ is an $n_i\times n_j$ matrix with entries in $\fraka_{ij}$.
    If $D$ is a division $F$-algebra and $(n_1,\dots,n_r)$ are natural numbers, let
    \[
    \calO_D^{[n_1,\dots,n_r]}=
    \left[
    \begin{array}{cccc}
    (\calO_D) & (\frakm_D) & \dots & (\frakm_D) \\
    \vdots & (\calO_D) & \ddots & \vdots \\
    \vdots & & \ddots & (\frakm_D) \\
    (\calO_D) &\dots  & \dots & (\calO_D)
    \end{array}
    \right]^{(n_1,\dots,n_r)} .
    \]

    \begin{thm}\label{TH:structure-of-hered-rings}
        Let $A$ be a hereditary $R$-order.
        Then there are  division $F$-algebras $\{D_i\}_{i=1}^t$ and integer
        tuples
        $\{\hat{n}^{(i)}=(n_1^{(i)},\dots,n_{r_i}^{(i)})\}_{i=1}^t$
        such that
        \[
        A\cong \prod_{i=1}^t \calO_{D_i}^{[\hat{n}^{(i)}]}\ .
        \]
        Conversely, every $A$ of this form is hereditary.
    \end{thm}

    \begin{proof}
        See \cite[Th.~39.14]{MaximalOrders} for the case where $A_F$  is a central simple
        $F$-algebra. The general case follows by using \cite[Th.~1.7.1]{HijNish94}, for instance.
    \end{proof}

\subsection{Projective Modules over Hereditary Orders}
\label{subsection:projectives}

    \emph{Keep the assumption that $R$ is a complete DVR.}
    We now collect several facts about projective modules over hereditary $R$-orders.

    \medskip

    We start with the following general lemma.

    \begin{lem}\label{LM:mod-jacob-radical}
        Let $A$ be a ring and let $P,Q\in\rproj{A}$.
        Write $\quo{P}=P/P\Jac(A)$ and $\quo{Q}=Q/Q\Jac(A)$.
        Then $P\cong Q$ if and only if $\quo{P}\cong \quo{Q}$ (as modules
        over $A$ or $\quo{A}=A/\Jac(A)$).
    \end{lem}

    \begin{proof}
        Using Nakayama's Lemma,
        it is easy to check that $P$ is a \emph{projective cover}
        of $\quo{P}$, and likewise, $Q$ is a projective cover of $\quo{Q}$.
        The lemma follows since projective covers are unique up to isomorphism.
    \end{proof}

    Let $D$ be a finite dimensional division $F$-algebra, let $\hat{m}=(m_1,\dots,m_r)$
    and let $A=\calO_D^{[\hat{m}]}$. It is easy to see
    that
    \[
    \Jac(\calO_D^{^{[\hat{m}]}})=
    \left[
    \begin{smallmatrix}
    (\frakm_D) & (\frakm_D) & {\text{\normalsize$\cdots$}} & (\frakm_D) \\[-5pt]
    (\calO_D) & (\frakm_D) & \ddots & \vdots \\[-5pt]
    \vdots & \ddots & \ddots & (\frakm_D) \\
    (\calO_D) & {\text{\normalsize$\cdots$}}  & (\calO_D) & (\frakm_D)
    \end{smallmatrix}
    \right]^{(m_1,\dots,m_r)}
    \]
    and hence $\quo{A}:=A/\Jac(A)\cong \nMat{k_D}{m_1}\times\dots\nMat{k_D}{m_r}$, where $k_D=\calO_D/\frakm_D$.
    For all $1\leq i\leq r$, write $\ell_i=m_1+\dots+{m_{i-1}}+1$,
    and let $e_i\in A$ denote the idempotent matrix with $1$ in the $(\ell_i,\ell_i)$-entry and $0$ in all
    other entries.
    Then $V_i:=e_iA$ is a projective right $A$-module such that $\quo{V_i}=\quo{e_iA}$ (notation as in Lemma~\ref{LM:mod-jacob-radical})
    is a simple $\quo{A}$-module. It is convenient to view  $V_i$ as the $\ell_i$-th row in the matrix
    presentation of $\calO_D^{[\hat{m}]}$, that is
    \[V_i= [\underbrace{\calO_D~\dots~\calO_D}_{m_1+\dots+m_{i}}~\underbrace{\frakm_D\dots \frakm_D}_{m_{i+1}+\dots+m_r}]\ ,\]
    where the action of $\calO_D^{[\hat{m}]}$ is given by matrix multiplication on the right.
    One easily checks that $\quo{V_1},\dots,\quo{V_r}$ is a complete list of simple $\quo{A}$-modules, up to isomorphism.
    Since any finitely generated $\quo{A}$-module $M$ is isomorphic to $\bigoplus_{i=1}^r \quo{V_i}^{n_i}$
    with $n_1,\dots,n_r\geq 0$ uniquely determined, Lemma \ref{LM:mod-jacob-radical} implies:

    \begin{prp}\label{PR:projective-decomp}
        In the previous setting, for every $P\in\rproj{\calO_D^{[\hat{m}]}}$,
        there are unique $n_1,\dots,n_r\geq 0$ such that $P\cong \bigoplus_{i=1}^rV_i^{n_i}$.
    \end{prp}

	\begin{remark}
		Proposition~\ref{PR:projective-decomp} can also be deduced by
		noting that  $A=\calO_D^{[\hat{m}]}$ is \emph{semiperfect}, a condition
		which implies
    	unique factorization of finitely generated projective $A$-modules;
    	see \cite[Th.~2.8.40, Pr.~2.9.21]{Ro88}.
	\end{remark}

\medskip

    Let $i,j\in\{1,\dots,r\}$. It is easy to see that $\Hom_A(V_i,V_j)=\Hom_A(e_iA,e_jA)\cong e_jAe_i$,
    where $e_jAe_i$ acts on $V_i=e_iA$ via multiplication on the left. Thus,
    \[
        \Hom_A(V_i,V_j)\cong e_j\calO_D^{[\hat{m}]}e_i\cong\left\{
        \begin{array}{ll}
        \calO_D & i\leq j \\
        \frakm_D & i>j
        \end{array}
        \right.\,.
    \]
    We therefore  identify $\Hom(V_j,V_i)$ with $\calO_D$ or $\frakm_D$.
    Notice that this identification turns composition into multiplication in $\calO_D$.

\remove{
    \begin{prp}\label{PR:serial}
        Let $A$ be a hereditary $R$-order and let $M$ be a simple
        $A_F$-module. Then the full $A$-lattices in $M$ form a chain with respect to inclusion.
    \end{prp}

    \begin{proof}
        Since for any two full $A$-lattices $L\subseteq L'$ in $M$, we have
        $\mathrm{length}(L'/L)<\infty$, it is enough to prove that any full $A$-lattice $L$
        in $M$ contains a unique maximal $A$-submodule.

        By Theorem~\ref{TH:structure-of-hered-rings}, we may assume $A=\calO_D^{[\hat{m}]}$
        as above. Let $L$ be a full $A$-lattice in $M$.
        Then $L\in\rproj{A}$ by Proposition~\ref{PR:lattices} and $L$ is indecomposable since $M$ is.
        Therefore, $L\cong V_i$ for some $1\leq i\leq r$.
        By Nakayama's Lemma, any maximal submodule of $V_i$ contains $V_i\Jac(A)$.
        However,  $\quo{V_i}=V_i/V_i\Jac(A)$ is simple for all $i$, so $L\cong V_i$
        contains a unique maximal $A$-submodule.
    \end{proof}
}

\subsection{Semilocal Rings}

    We finish this section by recording some useful facts
    about semilocal rings.

    \begin{prp}\label{PR:finite-alg-over-semilocal}
    	Let $R$ be a commutative semilocal ring. Then any $R$-algebra
    	$A$ that is finitely generated as an $R$-module is semilocal
    	and satisfies $A\Jac(R)\subseteq \Jac(A)$.
    \end{prp}

	\begin{proof}
		Let $J=\Jac(R)$. For all $a\in A$ and $r\in J$, we have
		$(1+ar)A+AJ=A$, so by Nakayama's Lemma, $(1+ar)A=A$. This implies
		that $1+AJ$ consists of right invertible elements, hence $AJ\subseteq \Jac(A)$. Next, $A/AJ$ is f.g.\ as an $R/J$-module,
		hence it is artinian. This means that
        $A/\Jac(A)$ is semisimple  (\cite[Th.~4.14]{Lam01}), so $A$ is semilocal.
	\end{proof}

    \begin{prp}\label{PR:etale-extsions-of-modules}
        Let $R$ be a commutative semilocal ring, let $A$ be an $R$-algebra that
        is finitely generated as an $R$-module, let $S$ be a faithfully flat commutative $R$-algebra,
        and let $P,Q\in\rproj{A}$. Then $P\cong Q$ if and only
        if $P_S\cong Q_S$ (as $A_S$-modules).
    \end{prp}

    \begin{proof}
		%
        Assume $P_S\cong Q_S$. Let $J=\Jac(R)$ and write $R/J$ as  a product of fields $\prod_{i=1}^tK_i$.
        We claim that $P_{K_i}\cong Q_{K_i}$ for all $1\leq i\leq t$.
        Indeed, $(P_S)\otimes_R K_i\cong P\otimes_R K_i\otimes_R S\cong P_{K_i}\otimes_{K_i} S_{K_i}$
        (as $S_{K_i}$-modules), and likewise $(Q_S)\otimes_R K_i\cong Q_{K_i}\otimes_{K_i} S_{K_i}$,
        so $P_{K_i}\otimes_{K_i} S_{K_i}\cong Q_{K_i}\otimes_{K_i} S_{K_i}$.
        Since $S/R$ is faithfully flat, $S_{K_i}\neq 0$, and hence there is a field
        $L$ admitting a nonzero morphism $S_{K_i}\to L$.
        Now, $P_{K_i}\otimes_{K_i}{L}\cong Q_{K_i}\otimes_{K_i} L$,
        so by \cite[Lm.~5.21]{BayFiMol13} (for instance), we get $P_{K_i}\cong Q_{K_i}$, as
        claimed.
        Finally, we have $P/PJ\cong\prod_iP_{K_i}\cong \prod_iQ_{K_i}\cong Q/QJ$.
        Since   $AJ\subseteq\Jac(A)$ (Proposition~\ref{PR:finite-alg-over-semilocal}),
        this means
        $P\cong Q$, by Lemma~\ref{LM:mod-jacob-radical}.
    \end{proof}

\section{Unimodular Hermitian Forms over Hereditary Orders}
\label{section:hermitian}

    Throughout, $R$ is a semilocal principal ideal domain (abbrev.: PID) with fraction field $F$.
    \emph{We  assume that $2\in\units{R}$.}
    The goal of this section is to prove:

    \begin{thm}\label{TH:forms-over-hereditary-has-genus-one}
        Let $A$ be a hereditary $R$-order,
        let $\sigma:A\to A$ be an $R$-involution, and let $u\in\Cent(A)$
        be an element with $u^\sigma u=1$.
        Let $P\in\rproj{A}$ and let $f,f':P\times P\to A$ be two unimodular $u$-hermitian forms
        over $(A,\sigma)$.
        Then $(P_F,f_F)\cong (P_F,f'_F)$ implies $(P,f)\cong (P,f')$.
    \end{thm}

\subsection{Hermitian Forms over Orders}

\label{subsection:herm-over-ord}

    As a preparation for the proof, we first recall the structure theory of hermitian forms over $R$-orders
    when \emph{$R$ is a complete DVR with $2\in\units{R}$}.
    This is a specialization of the  general theory in \cite[\S2--3]{QuSchSch79}.

    Throughout, $A$ denotes an $R$-order,
    $\sigma:A\to A$ is an $R$-involution, and $u\in\Cent(A)$ is an element
    satisfying $u^\sigma u=1$.
    Whenever it makes sense, an overline denotes reduction modulo $\Jac(A)$,
    e.g.\ $\quo{A}=A/\Jac(A)$ and $\quo{P}=P/P\Jac(A)$ for all $P\in\rproj{A}$.

\medskip

    Every $u$-hermitian space
    $(P,f)\in\Herm[u]{A,\sigma}$ gives rise to a $\quo{u}$-hermitian
    space $(\quo{P},\quo{f})\in\Herm[\quo{u}]{\quo{A},\quo{\sigma}}$
    defined by $\quo{f}(\quo{x},\quo{y})=\quo{f(x,y)}$.
    Since $R$ is a complete DVR, every finite $R$-algebra $E$ is semilocal and satisfies
    $E=\invlim \{E/\Jac(E)^n\}_{n\in\N}$ (see for instance \cite[p.~85]{MaximalOrders}).
    Therefore, well-known lifting arguments (\cite[Th.~2.2]{QuSchSch79}, note that $2\in\units{R}$) imply that:
    \begin{itemize}
        \item[(A)] $(\quo{P},\quo{f})\cong (\quo{P'},\quo{f'})$ if and only if
        $(P,f)\cong (P',f')$ and
        \item[(B)] every $\quo{u}$-hermitian space
        $(Q,g)\in\Herm[\quo{u}]{\quo{A},\quo{\sigma}}$ is isomorphic
        to  $(\quo{P},\quo{f})$ for some $(P,f)\in\Herm[u]{A,\sigma}$
    \end{itemize}

\medskip

	By Proposition~\ref{PR:finite-alg-over-semilocal},
    the ring $\quo{A}$ is semisimple and it is easy to see that  $\quo{\sigma}$
    permutes its simple factors. Therefore, we can
    write $(\quo{A},\quo{\sigma})=\prod_{i=1}^t (A_i,\sigma_i)$
    where for each $i$, either $A_i$ is simple artinian, or $A_i=B_i\times B_i^\op$
    and  $\sigma_i$ is the exchange involution. In the former case,
    we write $A_i=\nMat{W_i}{n_i}$ with $W_i$ a division ring.

\medskip

    We  decompose  $(\quo{P},\quo{f})$ as $\bigoplus_{i=1}^t(P_i,f_i)$ with $(P_i,f_i)\in\Herm[u_i]{A_i,\sigma_i}$
    (here, $\quo{u}=(u_i)_{i=1}^t$).
    In case $A_i=B_i\times B_i^\op$ and $\sigma_i$ is the exchange involution,
    the hermitian space $(P_i,f_i)$ is hyperbolic and  moreover determined up to isometry by the $A_i$-module $P_i$
    (see~\ref{subsection:hermitian-forms}).

    Suppose now that $A_i$ is simple.
    By \cite[Cor.~I.9.6.1]{Kn91} (for instance),
    there is an involution $\eta_i:W_i\to W_i$ and $\veps_i\in\Cent(W_i)$ with $\veps_i^{\eta_i}\veps_i=1$
    such that
    the category $\Herm[u_i]{A_i,\sigma_i}$ is equivalent to $\Herm[\veps_i]{W_i,\eta_i}$
    (this also induces an underlying equivalence between $\rproj{A_i}$ and $\rproj{W_i}$).
    Moreover, this equivalence is induced by an equivalence of the underlying \emph{hermitian categories}
    (\cite[II.3.4.2]{Kn91}), and hence preserves orthogonal sums and isotropicity.
    Here,  $(P_i,f_i)$ is istropic
    if $P_i$ has summand $N$ with $f_i(N,N)=0$; since $A_i$ is simple artinian,
    this is equivalent to the existence of $0\neq x\in P_i$ with $f_i(x,x)\neq 0$.

    Denote by $(Q_i,g_i)\in\Herm[\veps_i]{W_i,\eta_i}$ the hermitian space corresponding to $(P_i,f_i)$.
    We say that $(W_i,\eta_i,\veps_i)$ is of \emph{alternating type} if
    $W_i$ is a field, $\eta_i=\id_{W_i}$ and $\veps_i=-1$. In this case, $g_i$
    is just a nondegenerate alternating bilinear form, and hence it is hyperbolic and determined up to isomorphism by its base module
    $Q_i$
    (\cite[Th.~7.8.1]{SchQuadraticAndHermitianForms}).
    On the other hand, if $(W_i,\eta_i,\veps_i)$ is not of alternating type,
    then $g_i$ can be diagonalized (\cite[Th.~7.6.3]{SchQuadraticAndHermitianForms}).
    Furthermore, if $(Q_i,g_i)$ is isotropic, then we can  factor a hyperbolic plane
    from it (\cite[Lm.~7.7.2]{SchQuadraticAndHermitianForms}).
    (Recall that a hyperbolic plane is an isotropic unimodular $2$-dimensional
    $\veps_i$-hermitian space over $(W_i,\eta_i)$; it is always isomorphic to
    $(W_i\oplus W_i^*,\Hyp{W_i})$.)

\medskip

    We now draw some conclusions concerning the hermitian space $(P,f)$ using (A) and (B)
    above:
    The orthogonal decomposition $(\quo{P},\quo{f})=\bigoplus_{i=1}^t (P_i,f_i)$
    implies that
    we can write
    \[(P,f)=\bigoplus_{i=1}^t(P^{i},f^{i})\]
    with $(\quo{P^{i}},\quo{f^{i}})\cong (P_i,f_i)$.
    Furthermore, if $A_i=B_i\times B_i^\op$
    or $(W_i,\eta_i,\veps_i)$ is of alternating type,
    then $(P^i,f^i)$ is hyperbolic and its  isometry
    class  is determined by $P^i$.
    In all other cases, we can diagonalize $(P^{i},f^{i})$ in the sense that
    we can write $(P^{i},f^{i})=\bigoplus_{j=1}^{n_i}(V^{i},f^{i,j})$ where
    $V^{i}\in\rproj{A}$ is chosen such that $\quo{V^i}$ is a simple $A_i$-module
    ($V^i$ is uniquely determined up to isomorphism by Lemma~\ref{LM:mod-jacob-radical}).
    The induced decomposition $(\quo{P^i},\quo{f^i})=\bigoplus_j(\quo{V^{i}},\quo{f^{i,j}})$ corresponds
    to a diagonalization of $(Q_i,g_i)$.

\subsection{Proof of Theorem~\ref{TH:forms-over-hereditary-has-genus-one}}

    We  now prove Theorem~\ref{TH:forms-over-hereditary-has-genus-one}.
    The proof is done by a series of reductions to a simpler setting or an equivalent statement.
    Recall that we are given two unimodular $u$-hermitian forms $f$, $f'$ on an $A$-module $P$,
    and $A$ is a hereditary $R$-order.

\medskip

    {\noindent \it Reduction~1.} We may assume that $R$ is a complete DVR.

    \begin{proof}
    By {\cite[Th.~6.7]{BayFi14}}, $f\cong f'$
    if and only if $f_F\cong f'_F$ and
    $f_{\hat{R}_\frakp}\cong f'_{\hat{R}_\frakp}$
    for all $0\neq\frakp\in\Spec R$ (the notation is as in~\ref{subsection:orders}).
    It is therefore enough to prove that $f_{\hat{F}_\frakp}\cong f'_{\hat{F}_\frakp}$
    implies  $f_{\hat{R}_\frakp}\cong f'_{\hat{R}_\frakp}$.
    (Note that $A_{\hat{R}_\frakp}$ is hereditary  by Theorem~\ref{TH:local-hereditary}.)
    \end{proof}

    {\noindent \it Reduction~2.}
    We may assume that $A=\calO_D^{[\hat{m}]}$ (see~\ref{subsection:structure-of-hered-orders}),
    where $D$ is a f.d.\ division $F$-algebra and $\hat{m}=(m_1,\dots,m_r)$.

    \begin{proof}
    Since $R$ is a complete DVR (Reduction~1),
    Theorem~\ref{TH:structure-of-hered-rings} implies that there are f.d.\ division
    $F$-algebras $\{D_i\}_{i=1}^t$ and  integer tuples $\{\hat{n}^{(i)}\}_{i=1}^t$ such that
    $
    A\cong \prod_{i=1}^t \calO_{D_i}^{[\hat{n}^{(i)}]}
    $.
    It is easy to see that $\sigma$ permutes the components $\calO_{D_i}^{[\hat{n}^{(i)}]}$.
    Thus, $(A,\sigma)$ can be written as a product of rings with involution $\prod_{j=1}^s(E_j,\tau_j)$
    such that for each $j$, either $E_j=\calO_{D}^{[\hat{n}]}$,
    or $E_j=\calO_{D}^{[\hat{n}]}\times (\calO_{D}^{[\hat{n}]})^\op$
    and $\tau_j$ is the exchange involution.
    It enough to prove the theorem for each $(E_j,\tau_j)$ separately.
    However, when $E_j=\calO_{D}^{[\hat{n}]}\times (\calO_{D}^{[\hat{n}]})^\op$,
    all forms over $(E_j,\tau_j)$ are determined by their base module
    up to isomorphism (see~\ref{subsection:hermitian-forms}), so there is nothing to prove.
    \end{proof}

    {\it\noindent Notation~3.}
    We now apply all the notation of~\ref{subsection:herm-over-ord} to $(P,f)$
    and $(P,f')$. The objects induced by $f'$ will be written with a prime, e.g.\
    $(P_i,f'_i)$, $(Q_i,g'_i)$, etcetera.
    However, since $A/\Jac(A)\cong\nMat{k_D}{m_1}\times\dots
    \times\nMat{k_D}{m_r}$ with $k_D=\calO_D/\frakm_D$
    (see \ref{subsection:projectives}),
    we have $W_i=\calO_D/\frakm_D$ for all $i$, so we simply write
    $W$ instead of $W_i$.

    We further let $I$ be the set of indices $1\leq i\leq t$ (where $\quo{A}=\prod_{i=1}^tA_i$)
    for which  $A_i=B_i\times B_i^\op$ or $(W,\eta_i,\veps_i)$ is
    of alternating type, and let $J=\{1,\dots,t\}\setminus I$.

\medskip

    {\it \noindent Reduction~4.}
    We may assume that $P^i=0$ for all $i\in I$.

    \begin{proof}
        Write $(P^I,f^I)=\bigoplus_{i\in I}(P^i,f^i)$
        and define $(P^I,f'^I)$, $(P^J,f^J)$, $(P^J,f'^J)$ similarly.
        By~\ref{subsection:herm-over-ord}, the isometry classes of $f^I$ and $f'^I$
        are determined by $P$, so $f^I\cong f'^I$, and hence
        $f^I_F\cong f'^I_F$.
        Since $f_F\cong f'_F$ (by assumption), Theorem~\ref{TH:witt-cancellation} implies that $f^J_F\cong f'^J_F$.
        Now, if $f^J\cong f'^J$,
        then $f=f^J\oplus f^I\cong f'^J\oplus f'^I=f'$,
        so it is enough to prove that $f^J_F\cong f'^J_F$ implies
        $f^J\cong f'^J$.
    \end{proof}

    {\it \noindent Reduction 5.} We may assume that $(Q_i,g_i)$ is anisotropic for all $i\in J$
    (cf.\ Notation~3).

    \begin{proof}
        Fix some $i\in J$ and assume $(Q_i,g_i)$ is isotropic (so $Q_i\neq 0$).
        By~\ref{subsection:herm-over-ord} and the definition of $J$, the hermitian space $(Q_i,g'_i)$ is diagonalizable,
        hence we can write $(Q_i,g'_i)=(U_1,h_1)\oplus (U'_2,h'_2)$ with $\dim_W U_1=1$.
        As in~\ref{subsection:herm-over-ord}, this  induces a decomposition $(P^i,f'^i)=(U^1,h^1)\oplus (U'^2,h'^2)$.
        Since $(Q_i,g_i)$ is isotropic and $W$ is a division ring,
        we can factor a hyperbolic plane from $(Q_i,g_i)$ (cf.~\ref{subsection:herm-over-ord}).
        The space $(U_1,h_1)\perp (U_1,-h_1)$ is a hyperbolic plane (see \ref{subsection:herm-over-ord}),
        so $(U_1,h_1)$ is isomorphic to a summand of $(Q_i,g_i)$.
        Again, this induces a decomposition $(P^i,f^i)\cong (U^1,f^1)\oplus (U^2,h^2)$.
        Write $\tilde{f}=h^2\oplus\big(\bigoplus_{j\neq i}f^j\big)$
        and $\tilde{f}'=h'^2\oplus\big(\bigoplus_{j\neq i}f'^j\big)$.
        Then $f\cong h^1\oplus \tilde{f}$
        and $f'\cong h^1\oplus \tilde{f}'$. By arguing as
        in Reduction 4, we reduce into proving that $\tilde{f}_F\cong \tilde{f}'_F$
        implies $\tilde{f}\cong \tilde{f}'$.
        We repeat this until $(Q_i,g_i)$ is anisotropic for all $i\in J$.
    \end{proof}

    {\it \noindent Reduction 6.} It is enough to prove the following claim:
    \begin{enumerate}
    \item[($*$)] Let $(E,\tau)$ be a hereditary $R$-order with an $R$-involution.
    Then for all $a\in\units{\Sym}(E,\tau)$, we have $a\sim_{\tau_F} 1$ if and only if $a\sim_\tau 1$
    (notation as in~\ref{subsection:transfer}).
    \end{enumerate}
    under the following assumptions:
    \begin{enumerate}
        \item[(i)] $E=\calO_D^{[n_1,\dots,n_s]}$.
        \item[(ii)] Let $N=n_1+\dots+n_s$ and let $\{e_{ij}\}_{i,j}$ be the standard
        $D$-basis of $E_F=\nMat{D}{N}$. Then, $e_{ii}^\tau=e_{ii}$ for all $1\leq i\leq N$.
        \item[(iii)] Write $\quo{E}=E/\Jac(E)$. Then the induced involution $\quo{\tau}:\quo{E}\to\quo{E}$ is anisotropic, namely,
        $w^{\quo{\tau}}w\neq 0$ for any nonzero $w\in\quo{E}$.
    \end{enumerate}

    \begin{proof}
        By applying transfer with respect to $(P,f)$ as in Proposition~\ref{PR:transfer-into-edomorphism-ring}
        (see also Remark~\ref{RM:transfer-is-OK-with-scalar-ext}),
        we see that  Theorem~\ref{TH:forms-over-hereditary-has-genus-one} is equivalent to proving ($*$) for $E=\End_A(P)$ and $\tau=[g\mapsto \lAd{f}^{-1}g^*\lAd{f}]$.
        We shall verify that $E$ and $\tau$ satisfy (i), (ii) and (iii), given Reductions 1--5.

\smallskip

        For every $i\in J$ (cf.\ Notation~3), there is a unique $1\leq k_i\leq r$ such that $\quo{V^i}$
        of~\ref{subsection:herm-over-ord} (which is a simple $\quo{A}$-module) is isomorphic to $\quo{V_{k_i}}$, where $V_{k_i}$ is defined
        as in \ref{subsection:projectives}.
        We may therefore assume that $V^i=V_{k_i}$ (Lemma~\ref{LM:mod-jacob-radical}).
        Relabeling $J$, we may further assume that $J=\{1,\dots,s\}$ for some $1\leq s\leq t$
        and $i\leq j$ if and only if $k_i\leq k_j$.

        By~\ref{subsection:herm-over-ord} and Reduction~4, we can write
        $
        (P,f)
        =\bigoplus_{i=1}^s\bigoplus_{j=1}^{n_i}(V^i,f^{i,j})$ (note that $J=\{1,\dots,s\}$).
        In particular, $P\cong \bigoplus_{i=1}^s(V^i)^{\oplus n_i}$.
        Now, as explained in \ref{subsection:projectives},
        we have
        \[\Hom_A(V^i,V^j)=\Hom_A(V_{k_i},V_{k_j})\cong \left\{
        \begin{array}{ll}
        \calO_D & i\leq j \\
        \frakm_D & i>j
        \end{array}
        \right.\, ,
        \]
        and the isomorphism turns
        composition into multiplication in $\calO_D$. We therefore get
        \[
        \End_A\Big(\bigoplus_{i=1}^s(V^i)^{\oplus n_i}\Big)=\left[
        \begin{array}{ccc}
        (\Hom_A(V^1,V^1)) & \dots & (\Hom_A(V^s,V^1)) \\
        \vdots & & \vdots \\
        (\Hom_A(V^1,V^s)) & \dots & (\Hom_A(V^s,V^s))
        \end{array}
        \right]^{(\hat{n})}\cong\calO_D^{[\hat{n}]}\ ,
        \]
        for $\hat{n}=(n_1,\dots,n_s)$. This proves (i).

        Next, when identifying $\End_A(P)$ with $\calO_D^{[\hat{n}]}$ as above, the elements $e_{kk}\in \nMat{D}{N}$ of (ii)
        are orthogonal projections of $P$ onto a summand in the orthogonal decomposition
        $
        (P,f)=\bigoplus_{i=1}^s\bigoplus_{j=1}^{n_i}(V^i,f^{i,j})
        $. Therefore, $e_{kk}^{\tau}=e_{kk}$ by Remark~\ref{RM:summand-involution}.

        We finally show (iii):
        By reduction~5, the forms $g_1,\dots,g_s$ are anisotropic, and hence
        so are $f_1,\dots,f_s$ (see~\ref{subsection:herm-over-ord}).
        This means that $\quo{f}$ is anisotropic.
        By the proof of \cite[Pr.~3.3]{BayFi14} (for instance), we have $\Jac(E)P\subseteq P\Jac(A)$,
        and hence we can view $\quo{P}$ as  a left $\quo{E}$-module. Moreover, we have $\quo{E}=\End_{\quo{A}}(\quo{P})$.
        Using Proposition~\ref{PR:transfer-into-edomorphism-ring}
        and the definition of $\quo{f}$,
        it is easy to see that $\quo{f}(wx,y)=\quo{f}(x,w^{\quo{\tau}}y)$
        for every $x,y\in\quo{P}$ and $w\in\quo{E}$.
        Now, if $w^{\quo{\tau}}w=0$, then
        \[\quo{f}(w\quo{P},w\quo{P})=\quo{f}(\quo{P},w^{\quo{\tau}}w\quo{P})=0\ .\]
        Since $\quo{f}$ is anisotropic, we have $w\quo{P}=0$, so $w=0$ because $\quo{E}=\End_{\quo{A}}(\quo{P})$.
    \end{proof}

    The rest of the proof concerns with proving ($*$) under the assumptions (i)--(iii).

    \medskip

    {\it\noindent  Notation 7.}
    Write $\calO=\calO_D$, $\frakm=\frakm_D$ and $\nu=\nu_D$ (cf.\ Theorem~\ref{TH:structure-of-max-orders}).
    We view $\calO$ (resp.\ $D$) as a subring of $E$ (resp.\ $\nMat{D}{N}$)
    via the diagonal embedding.
    Scaling the additive discrete valuation $\nu$ if necessary,
    we may assume that $\nu(\units{D})=\Z$. We fix an element $\pi\in\units{D}$
    with $\nu(\pi)=1$.

    \medskip

    {\it\noindent Claim 8.} We have $s\leq 2$ (recall that $E=\calO_D^{[n_1,\dots,n_s]}$).

    \begin{proof}
        Suppose otherwise.
        Then by (i), there are $1\leq i<j<k\leq N$ such
        that
        \begin{align*}
        e_{ii}Ee_{jj}&=e_{ij}\frakm, & e_{ii}Ee_{kk}& =e_{ik}\frakm, & e_{jj}Ee_{kk}&=e_{jk}\frakm,\\
        e_{jj}Ee_{ii}&=e_{ji}\calO, & e_{kk}Ee_{ii}& =e_{ki}\calO, & e_{kk}Ee_{jj}&=e_{kj}\calO.
        \end{align*}
        However, by assumption (ii), we have
        \begin{align*}
        e_{ii}Ee_{kk}&=(e_{kk}Ee_{ii})^\tau=(e_{ki}\calO)^\tau=((e_{kj}\calO)(e_{ji}\calO))^\tau\\
        &=((e_{kk}Ee_{jj})(e_{jj}Ee_{ii}))^\tau=e_{ii}Ee_{jj}e_{jj}Ee_{kk}=e_{ij}\frakm e_{jk}\frakm=e_{ik}\frakm^2\,,
        \end{align*}
        so we have reached a contradiction.
    \end{proof}

    We now split into cases: When $s=0$, the ring $E$ is the zero ring, so there is nothing to prove.
    We proceed with the case $s=1$.

    \begin{proof}[Proof of {$(*)$} when $s=1$]
        Assume $a\sim_{\tau_F} 1$. Then there is $x\in E_F=\nMat{D}{N}$
        such that $x^\tau x=a$. Since $E=\nMat{\calO}{N}$ (because $s=1$),
        there is $m\in\Z$ such that $x\pi^m\in E$ and $\quo{x\pi^m}\neq 0$ in $\quo{E}$.
        If $m>0$, then $\quo{(x\pi^m)^\tau (x\pi^m)}=\quo{(\pi^m)^\tau x^\tau x\pi^m}=\quo{(\pi^m)^\tau a\pi^m}=
        \quo{(a\pi^m)^\tau}\cdot \quo{\pi^m}=0$,
        contradicting assumption (iii) in Reduction~6. Thus, $m\leq 0$, and we have $x\in E$ and $a\sim_\tau 1$.
    \end{proof}

    We assume henceforth that $s=2$, i.e.\ $E=\calO_D^{[n_1,n_2]}$.

\medskip

    {\it\noindent Claim~9.} For all $n\in\Z$, we have
    \[(\frakm^ne_{ij})^{\tau}=\left\{\begin{array}{ll}{\frakm^{n-1}e_{ji}} & i\leq n_1<j\\
    {\frakm^{n+1}e_{ji}} & i> n_1\geq j\\
    \frakm^ne_{ji} & \text{otherwise}\end{array}\right.\ ,\]
    where for $n<0$, we set $\frakm^n=\{x\in D\suchthat \nu(x)\geq n\}$.

    \begin{proof}
        For an ideal $I\lhd E$, we write $I^0=E$
        and $I^{-n}=\{x\in E_F\suchthat I^nxI^n\subseteq I^n\}$ ($n\geq 0$).
        It is routine to check that for all $n\in\Z$,
        \begin{equation}\label{EQ:powers-of-Jac-rad}
        \Jac(E)^{2n}=\left[\begin{array}{cc} (\frakm^{n}) & (\frakm^{n+1})\\ (\frakm^{n})& (\frakm^{n})\end{array}\right]^{(n_1,n_2)}
        \end{equation}
        (The case $n\geq 0$ can be shown by induction, and  then $n<0$ follows by computation; use the  valuation $\nu$
        on $D$.)

        The involution $\tau_F$ maps $\Jac(E)^{2n}$ bijectively onto itself for all $n\geq0$,
        and hence also for all $n<0$. Since
        \begin{equation}\label{EQ:U-ij}
        (e_{ji}D)^\tau=(e_{jj}E_Fe_{ii})^\tau=e_{ii}E_Fe_{jj}=e_{ij}D\ ,
        \end{equation}
        it follows that $\tau$ maps $e_{ij}D\cap \Jac(E)^{2n}$ bijectively onto
        $e_{ji}D\cap \Jac(E)^{2n}$ for all $n\in\Z$. Equation \eqref{EQ:powers-of-Jac-rad} now yields our claim.
    \end{proof}

    {\it\noindent Notation 10.}
    Write $U_{ij}=e_{ij}D=e_{ii}\nMat{D}{N}e_{jj}$. By \eqref{EQ:U-ij}, we have $U_{ij}^\tau=U_{ji}$.
    We extend $\nu$ to the spaces $U_{ij}$ by setting
    $
    \nu(de_{ij})=\nu(d)
    $ for all $d\in D$.
    If $x\in U_{ij}$ and $y\in U_{jk}$, then clearly
    \[\nu(xy)=\nu(x)+\nu(y)\ \]
    (note that $U_{ij}U_{jk}=U_{ik}$). In addition, by Claim~9, we have
    \begin{equation}\label{EQ:tau-nu-relations}
    \nu(x^{\tau})=\left\{\begin{array}{ll}{\nu(x)-1} & i\leq n_1<j\\
    {\nu(x)+1} & i> n_1\geq j\\
    \nu(x) & \text{otherwise}\end{array}\right.\ .
    \end{equation}

    \medskip

    {\it\noindent Claim 11.} Assume we are given $x_{ij}\in U_{ij}$ for all $1\leq i,j\leq N$.
    \begin{enumerate}
        \item[(a)] If $j\leq n_1$, then $\nu(\sum_{0<i\leq n_1}x_{ij}^\tau x_{ij})=\min_{0<i\leq n_1} \nu(x_{ij}^\tau x_{ij})$.
        \item[(b)] If $j\leq n_1$, then $\nu(\sum_{n_1<i\leq N}x_{ij}^\tau x_{ij})=\min_{n_1<i\leq N} \nu(x_{ij}^\tau x_{ij})$.
        \item[(c)] If $j> n_1$, then $\nu(\sum_{0<i\leq n_1}x_{ij}^\tau x_{ij})=\min_{0<i\leq n_1} \nu(x_{ij}^\tau x_{ij})$.
        \item[(d)] If $j> n_1$, then $\nu(\sum_{n_1<i\leq N}x_{ij}^\tau x_{ij})=\min_{n_1<i\leq N} \nu(x_{ij}^\tau x_{ij})$.
    \end{enumerate}

    \begin{proof}
        We first prove (a).
        Write $m=\min_{0<i\leq n_1}\nu(x_{ij})$
        and $x=\sum_{i=1}^{n_1}x_{ij}\pi^{-m}$. Note that $x\in E$,  $\quo{x}\neq 0$ (in $\quo{E}$)
        and $x^\tau x\in e_{jj}^{\tau}Ee_{jj}\subseteq U_{jj}$.
        We  have $\nu(\sum_{0<i\leq n_1}x_{ij}^\tau x_{ij})\geq \min_{0<i\leq n_1} \nu(x_{ij}^\tau x_{ij})$,
        and by \eqref{EQ:tau-nu-relations}, the right hand side equals $2m$.
        Assume by contradiction that $\nu(\sum_{0<i\leq n_1}x_{ij}^\tau x_{ij})>2m$.
        Then
        \[\nu(x^\tau x)=\nu\Big(\sum_{0<i\leq n_1}\sum_{0<i'\leq n_1}(\pi^{-m})^{\tau}x_{ij}^\tau x_{i'j}\pi^{-m}\Big)
        =\nu\Big(\sum_{0<i\leq n_1}x_{ij}^\tau x_{ij}\Big)-2m>0\ ,\]
        and so $\quo{x}^{\quo{\tau}} \quo{x}=0$, contradicting assumption (iii) of Reduction~6.
        Part (d) is shown in the same way.

        We now  prove (b).
        We have $\nu(\sum_{n_1<i\leq N}x_{ij}^\tau x_{ij})=\nu(\sum_{n_1<i\leq N}e_{jN}^\tau x_{ij}^\tau x_{ij}e_{jN})+1$
        (note that $\nu(e_{jN})=0$, and hence $\nu(e_{jN}^\tau)=-1$ by \eqref{EQ:tau-nu-relations}).
        By applying (d) to $x_{ij}e_{jN}\in U_{iN}$ ($n_1< i\leq N$), we get
        \[\nu\Big(\sum_{n_1<i\leq N}x_{ij}^\tau x_{ij}\Big)=\min_{n_1<i\leq N} \nu(e_{jN}^\tau x_{ij}^\tau x_{ij}e_{jN})+1=
        \min_{n_1<i\leq N} \nu(x_{ij}^\tau x_{ij})\ ,\]
        as required.
        Claim (c) is shown similarly (with (a) in place of (d)).
    \end{proof}

    We finally prove the remaining case $s=2$, thus completing the proof of Theorem~\ref{TH:forms-over-hereditary-has-genus-one}.

    \begin{proof}[Proof of $(*)$ in case $s=2$]
    Assume that
    $a\sim_{\tau_F} 1$. Then
    there is $x\in E_F$ such that $x^\tau x=a$.
    We claim that $x\in E$, and hence $a\sim_\tau 1$.

    Write $x=\sum_{i,j}x_{ij}$ with $x_{ij}\in U_{ij}$
    (cf.\ Notation~10), and
    fix some $0<j\leq n_1$.
    By parts (a) and (b) of Claim~11, we have
    \[\nu\Big(\sum_{0<i\leq n_1}x_{ij}^\tau x_{ij}\Big)=\min_{0<i\leq n_1} \nu(x_{ij}^\tau x_{ij})
    \quad\text{and}\quad
    \nu\Big(\sum_{n_1<i\leq N}x_{ij}^\tau x_{ij}\Big)=\min_{n_1<i\leq N} \nu(x_{ij}^\tau x_{ij})\ .\]
    By \eqref{EQ:tau-nu-relations}, $\nu(x_{ij}^\tau x_{ij})$ is even
    when $i\leq n_1$ and odd otherwise, so $\nu(\sum_{0<i\leq n_1}x_{ij}^\tau x_{ij})\neq \nu(\sum_{n_1<i\leq N}x_{ij}^\tau x_{ij})$.
    Thus,
    \[
    \nu\Big(\sum_{0<i\leq N}x_{ij}^\tau x_{ij}\Big)=
    \min\Big\{\nu\Big(\sum_{0<i\leq n_1}x_{ij}^\tau x_{ij}\Big), \nu\Big(\sum_{n_1<i\leq N}x_{ij}^\tau x_{ij}\Big)\Big\}=
    \min_{0<i\leq N} \nu(x_{ij}^\tau x_{ij})
    \]
    On the other hand, $\sum_{i=1}^Nx_{ij}^\tau x_{ij}=e_{jj}x^\tau xe_{jj}=e_{jj}ae_{jj}\in E$, so we
    have
    \[
    \min_{0<i\leq N} \nu(x_{ij}^\tau x_{ij})\geq 0\ .
    \]
    By \eqref{EQ:tau-nu-relations},
    we have $\nu(x_{ij}^\tau x_{ij})=2\nu(x_{ij})$ for $i\leq n_1$ and $\nu(x_{ij}^\tau x_{ij})=2\nu(x_{ij})+1$
    otherwise. Thus,
    $\nu(x_{ij})\geq 0$ for all $0\leq i<N$.

    Now fix some $n_1<j\leq N$. Using parts (c) and (d) of Claim~11, one similarly shows that
    \[
    \min_{0<i\leq N} \nu(x_{ij}^\tau x_{ij})\geq 0\ .
    \]
    In this case, $\nu(x_{ij}^\tau x_{ij})=2\nu(x_{ij})$ for $i> n_1$, and otherwise $\nu(x_{ij}^\tau x_{ij})=2\nu(x_{ij})-1$
    (by \eqref{EQ:tau-nu-relations}). Therefore,
    $\nu(x_{ij})\geq 0$ when $n_1<i$,
    and $\nu(x_{ij})\geq 1$ when $ i\leq n_1$.
    This means $x\in E=\calO_D^{[n_1,n_2]}$, as required.
    \end{proof}

\subsection{Corollaries and Remarks}

    We finish this section with some immediate corollaries and remarks.

    \begin{cor}\label{CR:unitary-group}
        Let $A$ and $\sigma$ be as in Theorem~\ref{TH:forms-over-hereditary-has-genus-one},
        and let $a,b\in \units{\Sym}(A,\sigma)$. If $a\sim_{\sigma_F} b$,
        then $a\sim_\sigma b$.
    \end{cor}

    \begin{proof}
        For $c\in \units{\Sym}(A,\sigma)$, define $f_c:A\times A\to A$ by $f_c(x,y)=x^\sigma cy$.
        It is easy to check that $f_c$ is a unimodular $1$-hermitian form over $(A,\sigma)$,
        and furthermore, $f_a\cong f_b$ if and only if $a\sim_\sigma b$. The corollary
        therefore follows from Theorem~\ref{TH:forms-over-hereditary-has-genus-one}.
    \end{proof}

    The following corollary will be needed in Section~\ref{section:non-unimodular}.
    We refer the reader to \cite[\S2]{BayFiMol13} for the relevant definitions (particularly the notion
    of scalar extension in hermitian categories).

    \begin{cor}\label{CR:main-for-herm-cats}
        Let $\catC$ be an \emph{$R$-linear hermitian category} (see \cite[\S2D]{BayFiMol13}) and let $P\in\catC$ be an object
        such that $\End_{\catC}(P)$ is a hereditary $R$-order.
        Let $f,f':P\to P^*$ be two unimodular $1$-hermitian forms. Then $(P_F,f_F)\cong (P_F,f'_F)$ implies
        $(P,f)\cong (P,f')$.
    \end{cor}

    \begin{proof}
        We reduce to the setting of Theorem~\ref{TH:forms-over-hereditary-has-genus-one}
        by applying \emph{transfer in hermitian categories} with respect to $(P,f)$; see for instance \cite[\S2C]{BayFiMol13}.
        Transfer is compatible with scalar extension by \cite[\S2E]{BayFiMol13}.
    \end{proof}


    \begin{remark}
        Let $A,\sigma,u,P,f,f'$ be as in Theorem~\ref{TH:forms-over-hereditary-has-genus-one} except the assumption that
        $A$ is hereditary.
        Then $f_F\cong f'_F$ implies that for every hereditary $R$-order $A\subseteq B\subseteq A_F$ with $B^\sigma =B$,
        we have $f_B\cong f'_B$. Here, $f_B:(P\otimes_AB)\times(P\otimes_AB)\to B$ is the $u$-hermitian form over $(B,\sigma_F|_B)$ given
        by $f_B(x\otimes b,x'\otimes b')=b^\sigma f(x,x')b'$, and $f'_B$ is defined similarly.
        Scharlau \cite[Th.~1]{Schar74A} proved that when $A_F$ is separable over $F$,
        one can always find a hereditary order $B$ as above.
        %
        %
        %
    \end{remark}

    \begin{remark}\label{RM:arbitrary-orders}
        Theorem~\ref{TH:forms-over-hereditary-has-genus-one} fails for non-hereditary
        orders; see~{\cite[Rm.~5.6]{BayFi14}}.

        On the other hand, the proof of Theorem~\ref{TH:forms-over-hereditary-has-genus-one}
        applies to many  $R$-orders with involution which are not hereditary.
        For example, assume $R$ is a DVR with maximal ideal $\frakm$,
        let $A=\smallSMatII{R}{\frakm^2}{R}{R}$, and define $\sigma:A\to A$ by
        $\smallSMatII{a}{b}{c}{d}^\sigma=\smallSMatII{d}{b}{c}{a}$.
        Then Step~1 can be applied to $(A,\sigma)$
        (\cite[Th.~6.2]{BayFi14})
		and
        one easily checks that $t=1$ and $I=\{1\}$ (cf.\ Notation~3),
        so after Reduction~4, we get $P=0$ and hence Theorem~\ref{TH:forms-over-hereditary-has-genus-one} holds
        for $u$-hermitian forms over $(A,\sigma)$.
        However, the example in \cite[Rm.~5.6]{BayFi14} shows that if $R$ is carefully chosen,
        then $A=\smallSMatII{R}{\frakm^2}{R}{R}$ has involutions for which the theorem fails.
    \end{remark}

    \begin{remark}
        In Theorem~\ref{TH:forms-over-hereditary-has-genus-one},
        the assumption that $f$ and $f'$ are defined
        on the same base module  (or on isomorphic $A$-modules) is necessary. For example,
        let $R$ be any  DVR with maximal ideal $\frakm$
        and consider
        \[
        A=\left[\begin{smallmatrix}
        R & \frakm & \frakm & \frakm \\
        R & R & \frakm & \frakm \\
        R & R & R & \frakm \\
        R & R & R & R
        \end{smallmatrix}\right]
        \]
        and the involution $\sigma:A\to A$ reflecting matrices
        along the diagonal emanating from the \emph{top-right} corner.
        Let $\{e_{ij}\}$ be the standard basis of $A_F=\nMat{F}{4}$, let
        \[
        P=(e_{11}+e_{44})A,\qquad P'=(e_{22}+e_{33})A,
        \]
        and define $1$-hermitian forms $f:P\times P\to A$, $f':P'\times P'\to A$
        by
        \[
        f(x,y)=x^\sigma y,\qquad f'(x',y')=x'^\sigma y'\ .
        \]
        Then $A$ is hereditary (Theorems~\ref{TH:local-hereditary} and~\ref{TH:structure-of-hered-rings}),
        and $x\mapsto (e_{21}+e_{34})x:P_F\to P'_F$ is easily seen
        to be an isometry from $(P_F,f_F)$ to $(P'_F,f'_F)$. However,
        $P$ and $P'$ are not isomorphic as $A$-modules, as can be easily seen
        by reducing modulo $\Jac(A)$ (in the sense of Lemma~\ref{LM:mod-jacob-radical}).
    \end{remark}

\section{Non-Unimodular Hermitian Forms over Hereditary Orders}
\label{section:non-unimodular}

    In this section, we use Theorem~\ref{TH:forms-over-hereditary-has-genus-one}
    and results from \cite{BayerFain96} to extend Theorem~\ref{TH:forms-over-hereditary-has-genus-one}
    to nearly unimodular hermitian forms.

\medskip

    Let $(A,\sigma)$ be a ring with involution and let $u\in\Cent(A)$ be an element
    satisfying $u^\sigma u=1$. Recall from the introduction that the \emph{coradical}
    of a $u$-hermitian form $f:P\times P\to A$ is defined as\footnote{
        We chose the name ``coradical'' because, in the literature, the kernel of $f_\ell$
        is often called the \emph{radical} of $f$.}
    \[
    \corad(f):=
    \mathrm{coker}(f_\ell:P\to P^*)\ .
    \]
    It is a right $A$-module, and
    $(P,f)$ is called \emph{nearly unimodular} if it is $A$-semisimple.

    It is easy to check that when $A$ is an $R$-algebra and $\sigma$ is $R$-linear,
    we have $\corad(f_S)\cong \corad(f)\otimes_RS$
    as $A_S$-modules  for any commutative $R$-algebra $S$ (e.g.\ use \cite[Lm.~1.2]{BayFi14}).


\medskip

    As in Section~\ref{section:hermitian}, assume henceforth
    that $R$ is a semilocal PID with $2\in\units{R}$, and let $F$ be the fraction field of $R$.
    We assume $R\neq F$.

    \begin{thm}\label{TH:non-unimodular-forms}
        Let $A$ be a hereditary $R$-order, let
        $\sigma:A\to A$ be an $R$-involution, and let $u\in\Cent(A)$ be an
        element with $u^\sigma u=1$.
        \begin{enumerate}[(i)]
            \item Let $P\in\rproj{A}$, and let $f,f':P\times P\to A$ be
            two nearly unimodular $u$-hermitian forms over $(A,\sigma)$ whose coradicals are isomorphic.
            Then $(P_F,f_F)\cong (P_F,f'_F)$ implies $(P,f)\cong (P,f')$.
            \item For any unimodular $u$-hermitian space $(Q,g)$ over $(A_F,\sigma_F)$
            there exists a nearly unimodular $(P,f)\in\herm[u]{A,\sigma}$ such that
            $(P_F,f_F)\cong (Q,g)$. Up to isomorphism, the number of such hermitian spaces is finite.
        \end{enumerate}
    \end{thm}

    When $A=R$ and $u=1$, part (i) of the theorem was proved
    by Auel, Parimala and Suresh \cite[Cor.~3.8]{AueParSur14} under the assumption
    that $\corad(f)$ is semisimple and cyclic. Part (ii) is a triviality in this setting.

    Scharlau \cite{Schar74A} showed that any separable $F$-algebra with an $F$-involution contains a
    hereditary $R$-order which is stable under the involution (see also \cite[Th.~1.7.1]{HijNish94}
    concerning orders in arbitrary algebras).
    This means that part (ii) of the theorem can be applied to any  separable $F$-algebra with involution.

\medskip

    We shall need several lemmas for the proof.
    For $0\neq\frakp\in\Spec R$, let $\hat{R}_\frakp$ and $\hat{F}_\frakp$
    denote the $\frakp$-adic completions of $R_\frakp$ and $F$, respectively.

    \begin{lem}\label{LM:simple-scalar-ext}
        Let $A$ be an $R$-order,  and let $M$ be a finitely generated right $A$-module.
        Then $M$ is semisimple if and only if $M_{\hat{R}_\frakp}$ is a semisimple
        $A_{\hat{R}_\frakp}$-module for all $0\neq\frakp\in\Spec R$. In this case,
        $M_F=0$.
    \end{lem}

    \begin{proof}
        To prove  ($\Longrightarrow$) and that $M_F=0$,
        we may assume $M$ is simple. Let $0\neq \frakp\in\Spec R$ ($\frakp$ exists since $R\neq F$). Then $M\frakp =M$ or $M\frakp =0$.
        When $M\frakp =M$, Nakayama's Lemma implies that $M$ is annihilated by an element
        of $1+\frakp$, hence $M_{\hat{R}_\frakp}=0$ and $M_F=0$. On the other hand, if $M\frakp=0$,
        then $M_F=0$, and
        the map $m\mapsto m\otimes 1:M\to M_{\hat{R}_\frakp}$ is an isomorphism of $A$-modules (its inverse
        is given by $m\otimes r\mapsto m r'$ where $r'$ is any element of $R$ with $r-r'\in\frakp\hat{R}_\frakp$).
        Thus, $M_{\hat{R}_\frakp}$ is simple as an $A$-module, and hence also as an $A_{\hat{R}_\frakp}$-module.

        To prove the other direction, it is enough to show that any surjection from $M$ to another right $A$-module
        is split, and
		this follows from \cite[Th.~3.20]{MaximalOrders} (this result treats localizations
		of $R$, but the proof generalizes verbatim to completions).
    \end{proof}

    \begin{lem}\label{LM:Jac-radical-condition}
    	Let $A$ be a hereditary $R$-order, let $J:=\Jac(A)$, and let $n\in\N$.
    	Then any a two-sided
    	$A$-lattice $L$ in $A_F$ satisfying $J^{n+1}L\subseteq A$
    	also satisfies $J^nLJ^n\subseteq A$.
    \end{lem}

    \begin{proof}
    	Let $J^{-1}=\{a\in A\suchthat aJ\subseteq A\}$. Since $A$ is hereditary,
    	$J^{-1}J=A$ (\cite[Th.~39.1]{MaximalOrders}). Now, $J^nLJ^n\subseteq J^{-1}(J^{n+1}L)J\subseteq J^{-1}AJ=A$.
    \end{proof}

    For the next lemmas,  let $\Mor(\rproj{A})$ denote  the category of morphisms
    in $\rproj{A}$. Recall that the objects of $\Mor(\rproj{A})$
    consist of triples $(P,f,Q)$ such that $P,Q\in\rproj{A}$
    and $f\in\Hom_A(P,Q)$. A morphism from $(P,f,Q)$ to $(P',f',Q')$
    is a pair $(\phi,\psi)\in\Hom_A(P,P')\times\Hom_A(Q,Q')$
    such that $f'\phi=\psi f$.

    \begin{lem}\label{LM:morphism-lemma}
        Let $A$ be any semilocal ring, and
        let
        $(P,f,Q),(P',f',Q')\in\Mor(\rproj{A})$.
        Then $(P,f,Q)\cong (P',f',Q')$ if and only if
        $P\cong P'$, $Q\cong Q'$ and $\coker(f)\cong \coker(f')$.
    \end{lem}

    \begin{proof}
        We only show the non-trivial direction.

        We first claim the following: Let $V,V'$ be isomorphic f.g.\ projective right $A$-modules,
        let $U, U'$ be arbitrary $A$-modules, let $\alpha,\alpha',\xi$
        be $A$-homomorphisms as in the diagram
        \[
        \xymatrix{
        V \ar[r]^{\alpha} \ar@{.>}[d]_{\psi} & U \ar[d]^{\xi} \\
        V' \ar[r]^{\alpha'} & U'
        }
        \]
        such that $\alpha$ and $\alpha'$ are surjective and $\xi$ is an isomorphism.
        Then there exists an isomorphism $\psi:V\to V'$ making the above diagram commutative.

        If the claim holds, then by taking $V=Q$, $V'=Q'$, $U=\coker(f)$, $U'=\coker(f')$
        and some isomorphism $\xi:U\to U'$,
        we get an isomorphism
        $\psi:Q\to Q'$ taking $\im(f)$ to $\im(f')$. Applying the claim again
        with $V=P$, $V'=P'$, $U=\im(f)$, $U'=\im(f')$, $\alpha=f$, $\alpha'=f'$, $\xi=\psi|_{\im(f)}$
        yields an isomorphism $\phi:P\to P'$ such that $\psi f=f'\phi$. Thus,
        $(\phi,\psi)$ is an isomorphism from $(P,f,Q)$ to $(P',f',Q')$.

        It is left to prove the claim: For any $A$-module $M$,
        write $\quo{M}=M/M\Jac(A)$ and let $\rho_M$ denote the projection
        $M\to \quo{M}$. The map $\alpha$ induces a surjective $A$-homomorphism
        $\quo{\alpha}:\quo{V}\to\quo{U}$. Since $\quo{A}$ is semisimple, we can write $\quo{V}=N\oplus \ker(\quo{\alpha})$
        and identify $N$ with $\quo{U}$ via $\quo{\alpha}$.
        We also write $W=\ker(\quo{\alpha})$ and let $\beta:\quo{V}=\quo{U}\oplus W\to W$
        denote the projection onto $W$. Consider the map
        $\eta:V\to U\oplus W$ given by $\eta(x)=\alpha x\oplus \beta(\rho_V{x})$.
        Observe that $\rho_V=(\rho_U\oplus \id_W)\circ \eta$,
        hence $\ker \eta \subseteq V\Jac(A)$.
        Since $\rho_V$ is also surjective, we have $U\oplus W=
        \im(\eta)+\ker(\rho_{U\oplus W})=\im(\eta)+(U\oplus W)\Jac(A)$,
        so by Nakayama's Lemma, $\eta$ is surjective ($U$ and $W$ are f.g.\ since they are epimorphic images of $V$).
        This means $\eta:V\to U\oplus W$ is a projective cover.
        In the same way, construct $\eta':V'\to U'\oplus W'$.
        Now, since $\quo{U}\oplus W=\quo{V}\cong\quo{V'}=\quo{U'}\oplus W'$
        and $\quo{U}\cong \quo{U'}$, there is an isomorphism
        $\zeta:W\to W'$ (because $\quo{A}$ is semisimple and $\quo{V}$ is f.g.). Consider the isomorphism $\xi\oplus \zeta:U\oplus W\to U'\oplus W'$.
        The universal property of projective covers implies that there is an isomorphism
        $\psi:V\to V'$ such that $(\xi\oplus \zeta)\eta=\eta'\psi$. Composing both sides with the projection
        $U'\oplus W'\to U'$ yields $ \xi\alpha=\alpha'\psi$,
        as required.
    \end{proof}

    \begin{lem}\label{LM:endo-ring-of-morph-with-semisimple-coker}
        Let $A$ be a hereditary $R$-order and let $(P,f,Q)\in\Mor(\rproj{A})$. If
        $f$ is injective and $\coker(f)$ is a semisimple $A$-module,
        then $\End_{\Mor(\rproj{A})}(P,f,Q)$ is a hereditary $R$-order.
    \end{lem}

    \begin{proof}
        The $R$-algebra $\End_{\Mor(\rproj{A})}(P,f,Q)$ is contained in $\End_R(Q)\times\End_R(P)$, so it is an $R$-order.
        It is not difficult to check that for any \emph{flat} $R$-algebra $S$,
        we have $\End(P_S,f_S,Q_S)\cong \End(P,f,Q)_S$ as $S$-algebras. Therefore, by Theorem~\ref{TH:local-hereditary}
        and Lemma~\ref{LM:simple-scalar-ext}, it is enough to prove the lemma when $R$ is a complete DVR.

        By Theorem~\ref{TH:structure-of-hered-rings}, $A\cong \prod_{i=1}^t\calO_{D_i}^{[\hat{n}_i]}$.
        Working in each component separately, we may assume $A=\calO_D^{[\hat{m}]}$ with $\hat{m}=(m_1,\dots,m_r)$. We
        now  use the notation introduced
        in \ref{subsection:projectives}, namely, the modules $V_1,\dots,V_r$ and the identification
        of $\Hom_A(V_i,V_j)$ with $\calO_D$ or $\frakm_D$.

        By the proof of \cite[Lm.~7.5]{BayFi14}, we can write $(P,f,Q)$
        as a direct sum of morphisms $\bigoplus_{j=1}^n(U_j,g_j,Z_j)$ such that for all
        $j$, either $Z_j=0$ and $U_j\neq 0$, or $Z_j$ is indecomposable and $g_j$ is injective.
        Since  $f$ is injective, $Z_j=0$ is impossible,
        so for all $j$, the module $Z_j$ is indecomposable and $g_j$ is injective.
        Furthermore, since $\coker(f)=\bigoplus_j\coker(g_j)$, the module $\coker(g_j)$ is semisimple
        for all $j$.

        Fix some $1\leq j\leq n$. There is unique $1\leq i\leq r$ such that $Z_j\cong V_i$ (Proposition~\ref{PR:projective-decomp}).
        Viewing $U_j$ as a submodule of $V_i$, we must have $V_i\Jac(A)\subseteq U_j \subseteq V_i$, because $V_i/U_j$ is semisimple.
        Since $V_i/V_i\Jac(A)$ is simple (see~\ref{subsection:projectives}), either $U_j=V_i$ or $U_j=V_i\Jac(A)$.
        In fact, $V_i\Jac(A)=V_{i-1}$ for $1< i \leq r$, and $V_1\Jac(A)\cong V_r$ via $x\mapsto \pi_D^{-1} x$,
        where $\pi_D$ is some generator of the $\calO_D$-ideal $\frakm_D$.
        It follows that $(U_j,g,Z_j)$ is isomorphic to
        \begin{itemize}
            \item $M_{2i-1}:=(V_i,1_D,V_i)$ for $1\leq i\leq r$, or
            \item $M_{2i}:=(V_{i},1_D,V_{i+1})$ for $1\leq i< r$, or
            \item $M_{2r}:=(V_r,\pi_D,V_1)$.
        \end{itemize}
        (recall that $\Hom_A(V_i,V_j)$ is identified with $\calO_D$ or $\frakm_D$).
        We may therefore write
        $
        (P,f,Q)\cong \bigoplus_{i=1}^{2r} M_i^{n_i}
        $.
        It is easy to check that for all $1\leq i, j\leq 2r$, we have
        \[
        \Hom(M_i,M_j)\cong\left\{ \begin{array}{ll}
        \calO_D & i\leq j \\
        \frakm_D & i>j
        \end{array}
        \right.\,.
        \]
        where the isomorphism is given by sending $(\phi,\psi)\in\Hom(M_i,M_j)$
        to $\phi$, viewed as an element of $\calO_D$ or $\frakm_D$.
        This isomorphism turns composition into multiplication in $\calO_D$.
        We now have
        \[
        \End(P,f,Q)=\left[
        \begin{array}{ccc}
        (\Hom(M_1,M_1)) & \dots & (\Hom(M_{2r},M_1)) \\
        \vdots & & \vdots \\
        (\Hom(M_1,M_{2r})) & \dots & (\Hom(M_{2r},M_{2r}))
        \end{array}
        \right]^{(\hat{n})}\cong\calO_D^{[\hat{n}]}
        \]
        where $\hat{n}=(n_1,\dots,n_{2r})$. Therefore, $\End(P,f,Q)$ is hereditary by Theorem~\ref{TH:structure-of-hered-rings}.
    \end{proof}

    We  now  prove Theorem~\ref{TH:non-unimodular-forms}.
    The proof uses \emph{$R$-linear hermitian categories} as defined in \cite[\S2D]{BayFiMol13}.
    Our notation will follow \cite[\S2]{BayFiMol13}, and we refer the reader to this source for all
    relevant definitions. See also \cite[Ch.~7]{SchQuadraticAndHermitianForms}
    or \cite[Ch.~II]{Kn91} for an extensive discussion.

    \begin{proof}[Proof of Theorem~\ref{TH:non-unimodular-forms}]
        (i) Recall that $u$-hermitian forms over $(A,\sigma)$ correspond to
        $1$-hermitian forms over the \emph{$R$-linear hermitian category} $(\rproj{A},*,\{\omega_P\}_{P\in\rproj{A}})$
        via $(P,f)\mapsto (P,f_\ell)$ (see \ref{subsection:hermitian-forms} for the definitions of $*$ and $\omega$).
        We make $\Mor(\rproj{A})$ into a hermitian category
        by setting $(P,f,Q)^*=(Q^*,f^*,P^*)$ and $\omega_{(P,f,Q)}=(\omega_P,\omega_Q)$
        (see \cite[\S3]{BayerFain96}); in fact, $\Mor(\rproj{A})$ is an $R$-linear hermitian
        category. By \cite[Th.~1]{BayerFain96},
        there is an equivalence between
        the category of \emph{arbitrary} $1$-hermitian forms over
        $\rproj{A}$ and the category of \emph{unimodular} $1$-hermitian forms over $\Mor(\rproj{A})$.
        This equivalence is compatible with \emph{flat} base change of $R$-linear hermitian categories
        (see~\cite[\S2D]{BayFiMol13} for the definition); the proof is similar to the proof of \cite[Pr.~3.7]{BayFiMol13}.
        Note that the base change in $\rproj{A}$, viewed as an $R$-linear hermitian category,
        is the same as the base change of finitely generated projective right $A$-modules
        by \cite[Rm.~2.2]{BayFiMol13}.

        Let $(M,h)$ and $(M',h')$ be the unimodular $1$-hermitian forms over $\Mor(\rproj{A})$ corresponding to
        $(P,f)$ and $(P,f')$, respectively. By the construction of the equivalence in  \cite[Th.~1]{BayerFain96},
        we have $M=(P,f_\ell,P^*)$ and $M'=(P,f'_\ell,P^*)$, so by Lemma~\ref{LM:morphism-lemma},
        the assumption $\corad(f)\cong \corad(f')$ implies
        that $M\cong M'$.
        Therefore, by the previous paragraph, the theorem will follow from
        Corollary~\ref{CR:main-for-herm-cats} if we show that $\End_{\Mor(\rproj{A})}(M)$ is hereditary.

        Since $\corad(f)$ is semisimple,  $\corad(f_F)\cong \corad(f)_F=0$ (Lemma~\ref{LM:simple-scalar-ext}).
        Thus, $(f_F)_\ell$ is onto. Since $A_F$ is semisimple (see \ref{subsection:orders}),
        $\mathrm{length}(P_F)=\mathrm{length}((P_F)^*)$, and hence $(f_F)_\ell$ is an isomorphism.
        This means that $f_\ell$ is injective. Now,
        $\End_{\Mor(\rproj{A})}(M)$ is a hereditary $R$-order by  Lemma~\ref{LM:endo-ring-of-morph-with-semisimple-coker}.


\vspace{0.2cm}

        (ii)
        For every full $A$-lattice $P$ in $Q$,
        let $\tilde{P}=\{x\in Q\suchthat g(P,x)\subseteq A\}$.
		Indentifying $Q$ with $Q':=\Hom_{A_F}(P_F,A_F)$ via $g_\ell$, we see
		that $\tilde{P}$ corresponds to the copy of $P'=\Hom_A(P,A)$ in $\Hom_{A_F}(P_F,A_F)$
		(see~\ref{subsection:orders}).
		Using this and Proposition~\ref{PR:lattices},
        it is easy to check that $\tilde{P}$ is a full $A$-lattice,
        and the map $P\mapsto \tilde{P}$ is involutive and reverses inclusion.
        Furthermore, $P\in\rproj{A}$, and
        if $P\subseteq\tilde{P}$, then $f:=g|_{P\times P}$ is a $u$-hermitian form over $(A,\sigma)$ and $\corad(f)\cong \tilde{P}/P$.
        It is therefore enough to prove that there is a full $A$-lattice $P$ in $Q$ such that $P\subseteq \tilde{P}$
        and $\tilde{P}/P$ is semisimple.

        Choose some full $A$-lattice $P$ in $Q$ and write $J=\Jac(A)$.
        Replacing $P$ with $P\cap \tilde{P}$, we may assume that $P\subseteq \tilde{P}$.
        The $A$-module $M:=\tilde{P}/P$ is of finite length, so
        by Nakayama's Lemma, there is $n\geq 0$, such that $MJ^n\neq 0$ and $MJ^{n+1}=0$.
        If $n=0$, then $M$ is semisimple (because $A$ is semilocal) and we are done,
        so assume $n>0$. Write $P_1=P+\tilde{P}J^{n}$.
        We claim that $P_1\subseteq \tilde{P}_1$.
        Provided this holds, we have
        $P\subsetneq P_1\subseteq \tilde{P}_1\subsetneq \tilde{P}$, and therefore we may replace $P$
        with $P_1$ and proceed by induction on the $A$-length of $M$.
		Proving $P_1\subseteq\tilde{P}_1$ is equivalent to showing $g(P_1,P_1)\subseteq A$.
		Write $L=g(\tilde{P},\tilde{P})$. Then $L$ is a  two-sided $A$-lattice in $A_F$.
		Furthermore, $J^{n+1}L=g(\tilde{P}J^{n+1},\tilde{P})\subseteq g(P,\tilde{P})\subseteq A$,
		so by Lemma~\ref{LM:Jac-radical-condition}, $J^{n}LJ^n\subseteq A$.
		Now, $g(P_1,P_1)=g(P,P)+g(\tilde{P}J^n,P)+g(P,\tilde{P}J^n)+g(\tilde{P}J^n,\tilde{P}J^n)
		\subseteq A+J^nLJ^n\subseteq A$, proving the claim.
		This completes the proof of the existence of $(P,f)$.

\smallskip

        It remains to show that,  up to isomorphism, there are finitely many $(P,f)\in\herm[u]{A,\sigma}$
        such that $(P_F,f_F)\cong (Q,g)$ and $\corad(f)$ is semisimple. By (i), it is enough to prove that there
        are only finitely many possibilities for
        $P$ and $\corad(f)$, up to isomorphism.

        We start with $P$. When $R$ is a complete DVR, Proposition~\ref{PR:projective-decomp}
        implies that there are finitely many $P$-s up to isomorphism with $P_F\cong Q$.
        For general $R$, note that $(P_{\hat{R}_\frakp})_{\hat{F}_\frakp}\cong Q\otimes_F{\hat{F}_\frakp}$
        as $A_{\hat{F}_\frakp}$-modules  for all $0\neq\frakp\in\Spec R$. Thus, by the case of a complete DVR,
        there are finitely many possibilities for $P_{\hat{R}_\frakp}\in\rproj{A_{\hat{R}_\frakp}}$, up to isomorphism.
        Since $\prod_{0\neq\frakp\in\Spec R}\hat{R}_\frakp$ is faithfully flat over $R$,
        Proposition~\ref{PR:etale-extsions-of-modules} implies that there are finitely many possible $P$-s up to isomorphism.

        To see that $\corad(f)$ has finitely many possibilities up to isomorphism,
        note that $\corad(f)$ is an epimorphic image of $P^*/P^*\Jac(A)$, which
        is semisimple of finite length. Since we showed that $P$ has finitely many
        possibilites up to isomorphism, we are done.
    \end{proof}

    \begin{remark}\label{RM:non-semisimple-corad}
        (i) Theorem~\ref{TH:non-unimodular-forms}(i) may fail
        for hermitian forms which are not nearly unimodular.
        For example, the quadratic forms $\Trings{1,9}$ and $\Trings{2,18}$
        are isomorphic over $\Q_3$
        (since $(x+3y)^2+9(\frac{1}{3}x-y)^2=2x^2+18y^2$), but not over
        $\Z_3$ (they are not equivalent modulo $3\Z_3$).
        Their coradicals are isomorphic to $\Z_3/9\Z_3$, which is not a semisimple
        $\Z_3$-module. There are also
        examples in which there is no similitude between the forms, e.g.\ $\Trings{1,1,9}$ and $\Trings{1,2,18}$  over $\Z_3$.

        (ii) The form $(P,f)$ in  Theorem~\ref{TH:maximal-orders}(ii) is not unique in general.
        For example, the quadratic forms $\Trings{1,1,-1}$ and $\Trings{1,3,-3}$ over $\Z_3$ are non-isomorphic and nearly
        unimodular, but they are isomorphic over $\Q_3$.
		
		(iii)
        The existence of $(P,f)$ in   Theorem~\ref{TH:non-unimodular-forms}(ii)  holds
        when
        $R$ is an arbitrary Dedekind domain; use \cite[Th.~4.21, Th.~4.22]{MaximalOrders} to reduce
        to the semilocal case.
    	
        (iv) If one allows hermitian spaces to have non-projective base modules, then
        the existence of $(P,f)$ in
    	Theorem~\ref{TH:non-unimodular-forms}(ii)  holds for
    	any $R$-order $A$ with Jacobson radical $J:=\Jac(A)$ satisfying:
    	\begin{itemize}
    		\item[($*$)] $J^2L\subseteq A$ implies $JLJ\subseteq A$
    		for any two-sided $A$-lattice $L$ in $A_F$.
    	\end{itemize}
        Indeed, in this case, Lemma~\ref{LM:Jac-radical-condition} holds for    $A$ (apply
        ($*$) to $J^{n-1}LJ^{n-1}$).
        Examples where ($*$) holds include all commutative $R$-orders. If $R$
        is a DVR with maximal ideal $\frakm$, then ($*$) also holds for the
        non-hereditary $R$-order $A=\smallSMatII{R}{\frakm^2}{R}{R}$
        considered in Remark~\ref{RM:arbitrary-orders}
        ($J=\smallSMatII{\frakm}{\frakm^2}{R}{\frakm}$ and
        the largest  $L\subseteq A_F$ with $J^2L\subseteq A$ is $\smallSMatII{\frakm^{-1}}{R}{\frakm^{-2}}{\frakm^{-1}}$).
        With the same notation, an example of $R$-order not satisfying ($*$) is given by taking
        \[
        A=\left[\begin{smallmatrix}
        R & \frakm & \frakm^2 \\
        R & R & \frakm \\
        R & R & R
        \end{smallmatrix}\right]
        \qquad
        \text{and}
        \qquad
        L=\left[\begin{smallmatrix}
        R & R & \frakm \\
        \frakm^{-1} & \frakm^{-1} & R \\
        \frakm^{-1} & \frakm^{-1} & R
        \end{smallmatrix}\right]\ .
        \]
        The details are left to the reader.
    \end{remark}

    Part (i) of Theorem~\ref{TH:non-unimodular-forms} can be strengthened when $A$ is assumed to be \emph{maximal}.

    \begin{thm}\label{TH:maximal-orders}
        Let $A,\sigma,u$ be as in Theorem~\ref{TH:non-unimodular-forms}, and suppose $A$ is a maximal
        $R$-order. Let $(P,f)$, $(P',f')$ be two nearly unimodular $u$-hermitian spaces over $(A,\sigma)$
        with isomorphic coradicals.
        Then $(P_F,f_F)\cong (P'_F,f'_F)$ implies $(P,f)\cong (P',f')$.
    \end{thm}

    The theorem
    follows from Theorem~\ref{TH:non-unimodular-forms}(i) and the following lemma:

    \begin{lem}\label{LM:rational-iso-of-modules}
        Let $A$ be  a maximal $R$-order and let $P,Q\in \rproj{A}$.
        Then $P_F\cong Q_F$ (as $A_F$-modules) if and only if $P\cong Q$.
    \end{lem}

    \begin{proof}
        We only prove the non-trivial direction.
        Suppose first that $R$ is a complete DVR. By Theorem~\ref{TH:separable-maximal-hereditary},
        $A$ is hereditary, so by Theorem~\ref{TH:structure-of-hered-rings},
        we may assume that $A=\prod_{i=1}^t\calO_{D_i}^{[\hat{n}^{(i)}]}$ (notation as in~\ref{subsection:structure-of-hered-orders}).
        Since $A$ is maximal, each of the tuples $\hat{n}^{(i)}$ must consist of a single integer, $n_i$,
        and hence $A=\prod_{i=1}^n\nMat{\calO_{D_i}}{n_i}$. By working componentwise,
        we may assume that $A=\nMat{\calO_D}{n}$ for a f.d.\ division $F$-algebra $D$.
        Furthermore, by {Morita Theory} (see \cite[\S18]{La99}), the categories
        $\rproj{A}$ and $\rproj{\calO_D}$ are equivalent, so we may further assume
        that $A=\calO_D$. Now, $A$ is local, so $P$ and $Q$ are free,
        say $P\cong A^n$ and $Q\cong A^m$. The assumption $P_F\cong Q_F$
        implies $n=m$ (because $A_F=D$ is a division ring), so $P\cong Q$.

        For general $R$, we have $P_{\hat{R}_\frakp}\cong Q_{\hat{R}_\frakp}$
        as $A_{\hat{R}_\frakp}$-modules
        by the previous paragraph (the $\hat{R}_{\frakp}$-order $A_{\hat{R}_{\frakp}}$ is
        maximal by Theorem~\ref{TH:local-hereditary}). Since $\prod_{0\neq\frakp\in\Spec R} \hat{R}_\frakp$
        is faithfully flat over $R$, Proposition~\ref{PR:etale-extsions-of-modules} implies
        that $P\cong Q$.
    \end{proof}

\section{A Cohomological Result}
\label{section:cohomological}

    In this section, we derive a cohomological result from Theorem~\ref{TH:forms-over-hereditary-has-genus-one} which
    is in the spirit of the Grothendieck--Serre conjecture (see the introduction).
    However, the algebraic groups involved are not necessarily reductive.
    In Section~\ref{section:BT}, we show that these group schemes are strongly related to group schemes
    constructed by Bruhat and Tits in \cite{BruhatTits84II}.


    Throughout, $R$ is a semilocal PID with $2\in\units{R}$ and
    $F$ is the fraction field of $R$. In addition,  $A$ is a \emph{hereditary} $R$-order,
    $\sigma:A\to A$ is an $R$-involution, and  $u\in\Cent(A)$ is an element satisfying $u^\sigma u=1$.
	Recall from \ref{subsection:orders} that $A_F$ is semisimple. For $\frakp\in\Spec R$, the fraction field
	of $R/\frakp$  is denoted $k(\frakp)$.

\medskip

    Let $(P,f), (P',f')\in\Herm[u]{A,\sigma}$.
	As usual, an $R$-algebra $S$ is called \emph{fppf} if $S$ is finitely presented as an $R$-algebra
	and
	flat as an $R$-module, and it is called \emph{\'etale} if
	in addition
	$S\otimes_R k(\frakp)$ is a finite product of separable field
	extensions of  $k(\frakp)$  for all $\frakp\in\Spec R$.
    We say that $(P',f')$ is an \emph{\'{e}tale form} (resp.\ \emph{fppf form}) of
    $(P,f)$ if there exists a faithfully flat \'{e}tale (resp.\ fppf) $R$-algebra $S$ such
    that $(P_S,f_S)\cong (P'_S,f'_S)$.
    The following propositions are well-known
    in the case $A=R$.

    \begin{prp}\label{PR:torsors}
        Fix $(P,f)\in\Herm[u]{A,\sigma}$ and let $\uU(f)$
        be the group scheme of isometries of $f$
        (see~\ref{subsection:hermitian-forms}). There is a one-to-one correspondence between:
        \begin{enumerate}
        	\item[{\rm(a)}] $\rmH^1_{\et}(R,\uU(f))$,
        	\item[{\rm(b)}] \'{e}tale forms of  $(P,f)$, considered up to isomorphism,
        	\item[{\rm(c)}] $\rmH^1_{\fppf}(R,\uU(f))$,
        	\item[{\rm(d)}] fppf forms of  $(P,f)$, considered up to isomorphism.
        \end{enumerate}
        This correspondence is compatible with scalar extension. Furthermore, the correspondence between {\rm(b)} and {\rm(d)}
        is given by mapping isomorphism classes to themselves, so any fppf form of $(P,f)$ is also an \'etale form.
    \end{prp}

    \begin{proof}
        The correspondence between (a) and (b), resp.\ (c) and (d),
        is standard and its proof follows the same lines as \cite[pp.\ 110--112, 117ff.]{Kn91}, for instance.
        The only additional thing to check is that faithfully flat descent of $A$-modules
        preserves the property of being finitely generated projective over $A$.
        To show this, one can argue as in \cite[Pr.~4.80(2)]{La99}; the proof  extends from $R$-modules
        to $A$-modules once
        noting that $\Hom_{A_S}(M_S,N_S)\cong \Hom_A(M,N)_S$ whenever $M$ is a finitely presented
        $A$-module and $S$ is a flat $R$-algebra (see for instance \cite[Th.~2.38]{MaximalOrders}).

        Upon identifying (a) with (b) and (c) with (d) as above, the map from (b) to (d) sending an isomorphism
        class to itself corresponds to the canonical map $\rmH^1_{\et}(R,\uU(f))\to \rmH^1_{\fppf}(R,\uU(f))$,
        and this map is an isomorphism because $\uU(f)$ is smooth over $ R$;
        see
        \cite[Apx.]{BayFi14} for the smoothness (note that $2\in\units{R}$)
        and \cite[Th.~11.7(1), Rm.~11.8(3)]{Groth68}
        for the isomorphism of the cohomologies.
    \end{proof}

    \begin{prp}\label{PR:etale-forms-crit}
    	Let $(P,f), (P',f')\in\Herm[u]{A,\sigma}$. Then $(P',f')$ is an \'etale (resp.\ fppf) form of $(P,f)$
    	if and only if $P\cong P'$.
    \end{prp}

    \begin{proof}
    	By Proposition~\ref{PR:torsors}, it is enough to prove the proposition for fppf forms.
    	The $(\Longrightarrow)$ direction follows from Proposition~\ref{PR:etale-extsions-of-modules},
    	and the $(\Longleftarrow)$ direction follows from  \cite[Pr.~A.1]{BayFi14} (note that $2\in\units{R}$).
    \end{proof}

	Using Propositions~\ref{PR:torsors} and~\ref{PR:etale-forms-crit}, we
	restate Theorem~\ref{TH:forms-over-hereditary-has-genus-one} in the language
	of \'{e}tale (or fppf) cohomology.
	Notice that  $A$ has to be hereditary
    (cf.\ Remark~\ref{RM:arbitrary-orders}).

    \begin{thm}\label{TH:main-coh}
        The map $\rmH^1_\et(R,\uU(f))\to \rmH^1_\et(F,\uU(f))$ is injective.
    \end{thm}


	We stress that the neutral component of $\uU(f)$,
	denoted $\uU(f)^0$, is not always reductive, so
    Theorem~\ref{TH:main-coh} does not follow from the Grothendieck--Serre conjecture.
	
	More precisely,
	by \cite[Apx.]{BayFi14}, $\uU(f)\to \Spec R$ is smooth
	and finitely presented,
	hence by \cite[Cor.~15.6.5]{Groth67EGAiv}
	(see also \cite[\S3.1]{Conrad14}), one may form $\uU(f)^0\to\Spec R$,
	the neutral
	component of $\uU(f)\to \Spec R$. It is
	the unique open  subscheme of $\uU(f)$ with the property
	that $(\uU(f)^0)_{k(\frakp)}=(\uU(f)_{k(\frakp)})^0$ for all $\frakp\in \Spec R$;
	here, the subscript $k(\frakp)$ denotes base change from $ R$ to $k(\frakp)$,
	and $(\uU(f)_{k(\frakp)})^0$ is the usual neutral component of the affine group $k(\frakp)$-scheme $\uU(f)_{k(\frakp)}$.
%
%
    According to \cite[Df.~3.1.1]{Conrad14}, a  group $R$-scheme
    $\bfG\to \Spec R$ is
    reductive if it is affine, smooth, and its geometric fibers
    are connected reductive algebraic groups
    (here we follow the convention that reductive group schemes are assumed to have connected fibers).
    However, the  example below shows that the closed fibers of $\uU(f)^0\to \Spec R$
    may be non-reductive.
	Analyzing precisely when this happens seems complicated.
	
    Nevertheless, we note that
    when $A_{k(\frakp)}$ is separable over $k(\frakp)$,
    the fiber $\uU(f)_{k(\frakp)}^0$
    is always a classical  reductive algebraic group over $k(\frakp)$.
    Indeed, using Remark~\ref{RM:transfer-is-OK-with-scalar-ext}, we find that $\uU(f)_{k(\frakp)}=\uU(f_{k(\frakp)})=\uU(E,\tau)$
    where $E=\End_{k(\frakp)}(P_{k(\frakp)})$ is a separable $k(\frakp)$-algebra and $\tau:E\to E$
    is a $k(\frakp)$-involution.
    Thus, when $A$ is a separable  $R$-order (cf.~\ref{subsection:orders}),
    the group scheme $\uU(f)^0\to \Spec R$ is reductive.
    When $A$ is a general hereditary order, the generic fiber
    $\uU(f_F)^0\to \Spec F$ is   \emph{pseudo-reductive},
    since $A_F$ is semisimple.

    \begin{example}\label{EX:reductive}
        Assume $R$ is a DVR with maximal ideal $\frakm=\pi R$
        and write $k=R/\frakm$. Let $A=\smallSMatII{R}{\frakm}{R}{R}$,
        and let $\sigma:A\to A$ be defined by $\sigma\smallSMatII{a}{\pi b}{c}{d}=\smallSMatII{a}{\pi c}{b}{d}$.
        Then $A$ is hereditary by Theorems~\ref{TH:local-hereditary} and~\ref{TH:structure-of-hered-rings}.
        Consider the $1$-hermitian form $f_1:A\times A\to A$ given by $f_1(x,y)=x^\sigma y$.
        It is easy to see that $\uU(f_1)\cong \uU(A,\sigma)$. The fiber  $\uU(A_F,\sigma_F)^0\to\Spec F$ is a
        well-known to be an $F$-torus of rank $1$, and hence reductive. However, $\uU(A_k,\sigma_k)^0\to\Spec k$ is not reductive.
        Indeed, as a $k$-algebra, $A_k=\smallSMatII{R/\frakm}{\frakm/\frakm^2}{R/\frakm}{R/\frakm}$ is isomorphic to $\nMat{k}{2}$ endowed with the
        multiplication $\smallSMatII{a}{b}{c}{d}\cdot\smallSMatII{x}{y}{z}{w}=
        \smallSMatII{ax}{ay+bw}{cx+dz}{dw}$,  and under this isomorphism, $\sigma_k$ becomes
        $\smallSMatII{a}{b}{c}{d}\mapsto\smallSMatII{a}{c}{b}{d}$.
        A straightforward computation now shows that
        $\uU(A_k,\sigma_k)^0$ is isomorphic to the additive group $\bfG_{\mathbf{a},k}$
        via $\smallSMatII{1}{b}{-b}{1}\mapsto b$ (on sections),
        so $\uU(A_k,\sigma_k)^0\to\Spec k$ is not reductive. In particular, $\uU(f_1)^0\to \Spec R$ is not reductive.

        On the other hand, if we replace $\sigma$ with the involution $\smallSMatII{a}{\pi b}{c}{d}\mapsto\smallSMatII{d}{\pi b}{c}{a}$,
        then a similar computation shows that  $\uU(f_1)^0\to \Spec R$ is reductive.
        In fact, the multiplicative group $\bfG_{\mathbf{m},R}$ is isomorphic to $\uU(A,\sigma)^0$ via
        $a\mapsto \smallSMatII{a}{0}{0}{a^{-1}}$ (on sections).
    \end{example}


\section{Relation with Bruhat--Tits Theory}
\label{section:BT}

	Let $R$, $F$, $A$, $\sigma$, $u$
	be as in Section~\ref{section:cohomological}, and assume that $A_F$ is a separable
    $F$-algebra.
    We also assume that  $R/\frakp$
	is perfect for all $0\neq \frakp\in\Spec R$.
	
	Let $(P,f)\in \Herm[\veps]{A,\sigma}$
	and let $\bfG$ denote the algebraic group $\uU(f_F)$ over $F$.
	Then by Theorem~\ref{TH:main-coh}, the map
	$
	\rmH^1_\et(R,\uU(f))\to \rmH^1_\et(F,\bfG)
	$	
	is injective.
	It is natural to ask for a characterization of
	the group $R$-schemes $\uU(f)$  which does not depend
	on the presentation of $\bfG$ as $\uU(f_F)$.
	In this section we provide such a characterization for the neutral component of
	$\uU(f)$ by relating it  with group schemes associated to   $\bfG$  by Bruhat and Tits.
    This suggests an extension of the Grothendieck--Serre conjecture for regular local rings of dimension $1$.

\medskip

	We begin by reformulating  the problem:
	By Remark~\ref{RM:transfer-is-OK-with-scalar-ext}, we have
	$\uU(f)\cong \uU(E,\tau)$,
	where $(E,\tau)$ is the  $R$-order with involution
	defined  in Proposition~\ref{PR:transfer-into-edomorphism-ring}.
	By applying Lemma~\ref{LM:endo-ring-of-morph-with-semisimple-coker} to $(P,\id,P)$, we see that $E$ is hereditary,
	and it is easy to see that $B:=E_F$ is a separable $F$-algebra.
	Conversely, if $E$ is any hereditary order in $B$ that is stable under $\tau$,
	then $\uU(E,\tau)\cong \uU(f_1)$ where $f_1:E\times E\to E$
	is the $1$-hermitian form given by $f_1(x,y)=x^\tau y$.
    %
	%
    %
	It is therefore enough to fix
	a separable $F$-algebra $B$, an $F$-involution $\tau:B\to B$, and consider  the group schemes $\uU(E,\tau)$
	as $E$ ranges over the $\tau$-stable hereditary orders in $B$.
	
\medskip
	
	We introduce further notation:
	Since $B$ is semisimple,
	we can factor $(B,\tau)$ as
	$\prod_{i=1}^t(B_i,\tau_i)$ where  $B_i$ is either simple artinian,
	or $B_i\cong B'_i\times B'^\op_i$ with $B'_i$ simple artinian and $\tau_i$
	exchanges $B'_i$ and $B'^\op_i$.
	
	Let $E$ be a $\tau$-stable hereditary order in $B$.
	By \cite[Th.~40.7]{MaximalOrders},  $E$ factors
	as  $\prod_{i=1}^tE_i$, where  $E_i$ is a hereditary $R$-order in $B_i$, hence $\uU(E,\tau)=\prod_i\uU(E_i,\tau_i)$.
	Fix some $1\leq i\leq t$.
	It is well-known that $\uU(B_i,\tau_i)$ is connected unless $B_i$ is simple and
	$\tau_i$ is an orthogonal involution. When $\tau_i$ is orthogonal,
	the neutral component $\uU(B_i,\tau_i)^0$
	is given as the scheme-theoretic kernel of the reduced norm map
	\[
	\mathrm{Nrd}_{B_i/K_i}:\uU(B_i,\tau_i)\to \calR_{K_i/F}\boldsymbol{\mu}_{2,K_i}\ ,
	\]
	where $K_i=\Cent(B_i)$ and
	$\calR_{K_i/F}$ is the Weil restriction from $K_i$ to $F$.
	The map $\mathrm{Nrd}_{B_i/K_i}$
	extends uniquely to a morphism $\uU(E_i,\tau_i)\to \calR_{S_i/R}\boldsymbol{\mu}_{2,S_i}$,
	where $S_i$ is the integral closure of $R$ in $K_i$,
	and we denote its scheme-theoretic kernel  by
	\[
	{\uU}(E_i,\tau_i)^\diamond\ .
	\]
	When $\tau_i$ is not orthogonal,  we define $\uU(E_i,\tau_i)^\diamond$
	to be $\uU(E_i,\tau_i)$. Finally, we set
	\[
	\uU(E,\tau)^\diamond =\prod_i \uU(E_i,\tau_i)^\diamond
	\]
	The group scheme
	$\uU(E,\tau)^\diamond$ is open in $\uU(E,\tau)$, hence it is smooth
	over $\Spec R$. It is in general larger than $\uU(E,\tau)^0$.

\medskip

	We now recall some  facts from the works of Bruhat and Tits on
	reductive algebraic groups over valuated fields.
	\emph{Throughout the discussion,  $R$ is a  henselian DVR
	and $\bfG$ is a (connected) reductive algebraic group over $F$.}
	The strict henselization of $R$ is denoted $R^\sh$
	and its fraction field is $F^\sh$.
	Our standing assumption that the residue field
    of $R$ is perfect is necessary for some of the facts that we shall cite, and also saves  some
    technicalities.
	
	In \cite{BruhatTits72I} (see also  also
	\cite{Tits79}),
	Bruhat and Tits associate with $\bfG$ a metric space $\calB(\bfG,F)$,
	on which $\bfG(F)$ acts via isometries,
	called the   \emph{extended affine Bruhat--Tits building} of $\bfG$.\footnote{
		We alert the reader that many texts also  consider the \emph{non-extended} building of $\bfG$.
		In this paper, however, the term ``building'' always means ``extended building''. The distinction between
		these two concepts is unnecessary when $\bfG$ is semisimple.
	}
	The formation of $\calB(\bfG,F)$ is functorial relative to Galois extensions in the sense
	that if $K/F$ is a Galois extension, then $\calB(\bfG,F)$ embeds in  $\calB(\bfG,K)$ as $\bfG(F)$-sets,
	$\Gal(K/F)$ acts isometrically on $\calB(\bfG,K)$ while fixing $\calB(\bfG,F)$, and when $K/F$ is unramified,
	the fixed point set of $\Gal(K/F)$ is precisely $\calB(\bfG,F)$ (\cite[\S2.6]{Tits79}).
	Furthermore, any automorphism of $\bfG$ gives rise to an automorphism
    of $\calB(\bfG,F)$, and if $\bfG=\bfG_1\times\bfG_2$, then $\calB(\bfG,F)= \calB(\bfG_1,F)\times\calB(\bfG_2,F)$.

	The building $\calB(\bfG,F)$ carries a partition into  facets, which is respected by the action of $\bfG(F)$.
	More precisely, when $\bfG$ is simple, $\calB(\bfG,F)$ has the structure of a simplicial
	complex, whereas in general, letting   $\bfG_1,\dots,\bfG_r$ denote the absolutely simple factors
	of $\bfG^0$ and $s$ be the split rank of the center of $\bfG$,
	we have $\calB(\bfG,F)\cong \calB(\bfG_1,F)\times\dots\times\calB(\bfG_r,F)\times \R^s$, and a facet of $\calB(\bfG,F)$
	consists of a Cartesian product $C_1\times\dots\times C_r\times \R^s$ with $C_i$ a facet of $\calB(\bfG_i,F)$.
    In addition, any two points in $\calB(\bfG,F)$ can be joined by a unique geodesic segment; see \cite[\S2.2]{Tits79}.

	For every $y\in\calB(\bfG,F)$, write
	$
	G_y=\Fix_{\bfG(F)}(y):=\{g\in \bfG(f)\suchthat gy=y\}
	$.
	Replacing $F$ with $F^\sh$, we define $\tilde{G}_y\subseteq \bfG(F^\sh)$
	similarly.
	We note that $\tilde{G}_y$ determines the facet $C$ of $\calB(\bfG,F^\sh)$ containing $y$; it is the unique
	facet containing all points fixed by $\tilde{G}_y$ (\cite[Cor.~5.1.39]{BruhatTits84II} or \cite[\S 1.7]{BruhatTits87III}).
	By \cite[\S{}3.4.1]{Tits79}, there exists
	a  smooth affine group $R$-scheme
	$\catG_y$
    whose generic fiber is $\bfG$  such that $\catG_y(\tilde{R})=\tilde{G}_y$;
    these properties determine $\catG_y$
    up to an isomorphism respecting the identification $\catG_{y,F}=\bfG$
    (cf.\ \cite[1.7.6, {1.7.3(\mbox{a-a1})}]{BruhatTits84II}).
    We call $\catG_y$ a \emph{point stabilizer} group scheme of $\bfG$.
    The groups $\catG_y^0(R)$ are known as the \emph{parahoric subgroups} of $\bfG(F)$ (\cite[Df.~5.2.6]{BruhatTits84II}).
    We therefore call $\catG_y^0$ a \emph{parahoric} group scheme of $\bfG$.

\medskip

	We now give an alternative description
	of the groups $\catG_y$ when $\bfG=\uU(B,\tau)^0$.

	\begin{thm}\label{TH:parahoric}
		The point stabilizer group schemes of $\bfG:=\uU(B,\tau)^0$
		are  the group schemes  $\uU(E,\tau)^\diamond$
		where $E$ ranges over the $\tau$-stable hereditary $R$-orders in $B$.
    \end{thm}

    Our proof is based on applying a result of Prasad and Yu  \cite[Th.~1.9]{PrasYu02} to the following theorem of Bruhat and Tits.
    The special case of \cite[Th.~1.9]{PrasYu02} that we need also follows implicitly from results in \cite{BruhatTits87SU}.

    \begin{thm}[Bruhat, Tits]\label{TH:building-GLn}
    	The point stabilizer group schemes of
		$\bfH:=\uGL_1(B)$ are  the groups $\uGL_1(E)$  where $E$
		ranges over the hereditary $R$-orders in $B$.
    \end{thm}

    \begin{proof}
    	Factorizing $B$ as a product of simple artinian $F$-algebras and working in each factor separately
    	(using \cite[Th.~40.7]{MaximalOrders}),
    	we may assume that $B=\nMat{D}{n}$ where $D$ is a separable division $F$-algebra.
    	
    	By  \cite[Th.~2.11]{BruhatTits84GLD}, $\calB(\bfH,F)$ can be identified  with the collection of \emph{splittable norms}
		on $D^n$ (see \cite[Df.~1.4]{BruhatTits84GLD} for the definition).\footnote{
			The assumption that $\Cent(D)=F$ in \cite{BruhatTits84GLD}
			can be ignored by viewing $\uGL_1(\nMat{D}{n})$ as a group scheme
			over $K=\Cent(D)$ and using  fact that the building of
			$\bfH\to \Spec K$ is canonically isomorphic to the building of the Weil restriction
			$\calR_{K/F}\bfH\to \Spec F$ (\cite[p.~44]{Tits79}).}
		From sections 1.17, 1.23 and 1.24
		of \cite{BruhatTits84GLD}, it follows
		that the $\bfH(F)$-stabilizers of points in $\calB(\bfH,F)$,
		are precisely the sets $\units{E}$
		as $E$ ranges over the hereditary orders in $\nMat{D}{n}$.
		Furthermore, by \cite[Th.~4.7]{BruhatTits84GLD} (see also section 2.14 there),
        given an unramified Galois extension $K/F$, a hereditary order $E$ in $B$
        and $y\in\calB(\bfH,F)$ with $\units{E}=\Fix_{\bfH(F)}(y)$, we have $\units{(E\otimes S)}=\Fix_{\bfH(K)}(y)$,
		where $S$ is the integral closure of $R$ in $K$.
		Taking the limit over all unramified Galois  extensions, we see that $\units{(E\otimes R^{\sh})}=\Fix_{\bfH(K^\sh)}(y)$.
		Since $\uGL_1(E)$ is a smooth affine group $R$-scheme  (being an open subscheme of $\mathbb{A}^{\dim B}_R$),
		and since $\uGL_1(E)(R^{\sh})=\units{(E\otimes R^{\sh})}$, the point stabilizer group scheme
        associated to $y$ must be $\uGL_1(E)$.
    \end{proof}


	\begin{proof}[Proof of Theorem~\ref{TH:parahoric}]
    	Write $(B,\tau)=\prod_i(B_i,\tau_i)$ as above. It is enough to prove the theorem
    	for each of the factors separately. We may therefore assume that
    	$B$ is simple artinian, or $B=B'\times B'^\op$ with $B'$ simple aritinian
    	and $\tau$ is given by $(a,b^\op)\mapsto (b,a^\op)$.
    	
    	In the latter case, we have $\uU(B,\tau)^0=\uU(B,\tau)\cong \uGL_1(B')$ via $(x,y^\op)\mapsto x$ on sections.
		Since
		any $\tau$-stable hereditary order $E$ in $B$ is of the form
		$E=E'\times E'^\op$ with $E'$  a hereditary order in $B'$, Theorem~\ref{TH:building-GLn}
		implies that $\uU(E,\tau)^\diamond=\uU(E,\tau)\cong \uGL_1(E')$ is a point stabilizer group scheme of $\bfG$,
		and all point stabilizer group schemes  are obtained in this manner.
		
		Suppose henceforth that $B$ is simple  and write $\bfH=\uGL_1(B)$.
		Consider the automorphism $\tilde{\tau}:\bfH\to\bfH$ given by $x\mapsto (x^{-1})^{\tau}$ on sections.
		Then
		$\uU(B,\tau)$ is the group scheme of $\tilde{\tau}$-fixed points in $\bfH$.
		The automorphism $\tilde{\tau}$
		induces an automorphism on the building $\tilde{\tau}:\calB(\bfH,F^{\sh})\to\calB(\bfH,F^{\sh})$
		satisfying $\tilde{\tau}(gy)=\tilde{\tau}(g)\tilde{\tau}(y)$ for all $g\in\bfH(F^{\sh})$, $y\in\calB(\bfH,F^\sh)$.
		A theorem of Prasad and Yu  \cite[Th.~1.9]{PrasYu02}
		now asserts that the space of $\tilde{\tau}$-fixed points in $\calB(\bfH,F^{\sh})$ is  isomorphic to $\calB(\bfG,F^{\sh})$
		both as  $\bfG(F^{\sh})$-sets and as $\Gal(F^{\sh}/F)$-sets.
		
		Let $E$ be a $\tau$-stable hereditary
		$R$-order in $B$. By Theorem~\ref{TH:building-GLn}, there exists  $y\in\calB(\bfH,F)$
		whose point stabilizer group scheme relative to $\bfH$ is $\uGL_1(E)$,
		hence $\Fix_{\bfH(F^{\sh})}(y)=\units{(E\otimes R^{\sh})}$.
		As $\tilde{\tau}(\units{E})=\units{E}$,
		we also have $\Fix_{\bfH(F^{\sh})}(\tilde{\tau}(y))=\units{(E\otimes R^{\sh})}$.
		Since the fixer of a point determines the facet containing it, $y$ and $\tilde{\tau}(y)$
		are contained in the same facet $C$ of $\calB(\bfH,F^{\sh})$.
		Let $z$ be middle point   of the geodesic segment connecting $y$ and $\tilde{\tau}(y)$.
		Then $z\in C$, $\tilde{\tau}(z)=z$, and $z$ is invariant under $\Gal(F^{\sh}/F)$,
		hence $z\in\calB(\bfG,F)$.
		By section 3.6 of \cite{BruhatTits84GLD}, the fixer of a
		point of $\calB(\bfH,F^{\sh})$ depends only on
		the facet containing it, hence $\Fix_{\bfH(F^{\sh})}(z)=\Fix_{\bfH(F^{\sh})}(y)=\units{(E\otimes R^{\sh})}$.
		Since $\tilde{\tau}(z)=z$,
		the fixer of $z$ in $\bfG(F^{\sh})$ is $\units{(E\otimes R^{\sh})}\cap \uU(B,\tau)^0(F^\sh)=\uU(E,\tau)^{\diamond}(R^{\sh})$.
		Since $\uU(E,\tau)^\diamond\to \Spec R$ is affine and smooth, it must be the point stabilizer  group scheme
		$\catG_z$.
		
		Conversely, let $z\in\calB(\bfG,F)$. By Theorem~\ref{TH:building-GLn}, there is a hereditary $R$-order $E$
		with $\Fix_{\bfH(F)}(z)=\units{E}$. Since $\tilde{\tau}(z)=z$,
		we have $\tau(\units{E})=\units{E}$, which implies that $E$ is stable under $\tau$ (since
		$\units{E}$ generates $E$ as an additive group whenever $|R/\Jac(R)|>2$; the proof, using
		upper and lower triangular matrices
		in $E/\Jac(E)$,
		is omitted). Now,
		as in the previous paragraph, we get $\catG_z=\uU(E,\tau)^\diamond$.
    \end{proof}

    \emph{We now retain our original setting where $R$ is a semilocal PID.}

\medskip

	Let $\bfG$ be a reductive algebraic group over $F$
	and let $\catG\to \Spec R$ be a group scheme with $\catG_F=\bfG$. Let
	us say that $\catG$ is a \emph{point stabilizer} group scheme of $\bfG$ if
	$\catG$ is affine, smooth and for every $0\neq \frakp\in\Spec R$,
	the group scheme $\catG_{\hat{R}_\frakp}$ is a point stabilizer group scheme
	of $\bfG_{\hat{F}_\frakp}$ as above ($\hat{R}_\frakp$, $\hat{F}_\frakp$
    are defined as in Section~\ref{section:non-unimodular}).
	In this case, call $\catG^0$ a \emph{parahoric} group
	scheme of $\bfG$.
	Theorem~\ref{TH:parahoric} implies:

    \begin{cor}\label{CR:parahoric}
    	The point stabilizer   group schemes of $\uU(B,\tau)^0$
    	are precisely the group schemes  $\uU(E,\tau)^\diamond$
		where $E$ ranges over the $\tau$-stable hereditary $R$-orders in $B$.
    \end{cor}

    \begin{proof}
    	Theorem~\ref{TH:local-hereditary} implies that
    	that $\uU(E,\tau)^\diamond$   is a point stabilizer group scheme
    	for any $\tau$-stable hereditary $R$-order $E$, so we need to show the converse.
    	
    	If $\catG$ is a point stabilizer group scheme, then for any $0\neq\frakp\in \Spec R$, there
    	is a $\tau$-stable $\hat{R}_\frakp$-order $E_\frakp$ in $B_{\hat{R}_\frakp}$ such that
    	$\catG_{\hat{R}_\frakp}\cong \uU(E_\frakp,\tau)^\diamond$ and the isomorphism
    	extends the isomorphism $\catG_{\hat{F}_\frakp}\cong\uU(B,\tau)^0_{\hat{F}_\frakp}$ induced
    	by $\catG_F=\uU(B,\tau)^0$.
    	Embedding $B$ diagonally in $\prod_\frakp B_{\hat{R}_\frakp}$, let $E=B\cap\prod_\frakp E_\frakp$.
    	Then $E$ is an $R$-order in $B$ with $E_{\hat{R}_\frakp}=E_\frakp$ (\cite[Th.~5.3]{MaximalOrders}), hence $E$ is hereditary
    	by  Theorem~\ref{TH:local-hereditary}. We claim that the identification $\catG_F= \uU(B,\tau)^0$
    	extends to an isomorphism $\catG\cong \uU(E,\tau)$.
    	
    	Write $\catG=\Spec S$, and for any $R$-algebra $R'$, let $R'[\catG]=S\otimes_RR'$.
    	Similar notation will be applied to all affine schemes.
    	Since $\catG$ and $\catU:=\uU(E,\tau)^\diamond$ are flat over $R$, we may view $R[\catG]$ and $R[\catU]$
    	as subrings of $F[\uU(B,\tau)^0]=F[\catG]=F[\catU]$.
        Likewise, for every $0\neq \frakp\in\Spec R$, 
        we may regard $\hat{R}_\frakp[\catG]$ and $\hat{R}_\frakp[\catU]$ as subrings of $\hat{F}_\frakp[\uU(B,\tau)^0]$.
        In fact, by the previous paragraph, $\hat{R}_\frakp[\catG]=\hat{R}_\frakp[\catU]$.
        %
        Write $M=R[\catG]+R[\catU]\subseteq F[\catG]$. Then the inclusion
        $R[\catG]\to M$ becomes an isomorphism after extending scalars to $\hat{R}_\frakp$ for all $0\neq \frakp\in\Spec R$.
        Since $\prod_{\frakp\neq 0} \hat{R}_\frakp$ is a faithfully flat $R$-module,
        $R[\catG]=M$, and likewise $R[\catU]=M$. It follows that
    	$R[\catG]=R[\catU]$, namely, the isomorphism $\catG_F\cong \catU_F=\uU(B,\tau)^0$
    	extends to an isomorphism $\catG\cong \catU$.
    \end{proof}

    Corollary~\ref{CR:parahoric} and Theorem~\ref{TH:main-coh} suggest the following question, which
    extends the Grothendieck--Serre conjecture (see the introduction)
    for regular local rings of dimension $1$  with  perfect residue field.

    \begin{que}\label{QU:GS}
    	Let $\catG\to \Spec R$ be a group scheme such that $\bfG:=\catG_F$ is reductive.
    	Is the base change map
    	\[
    	\rmH^1_{\et}(R,\catG)\to \rmH^1_{\et}(F,\bfG)
    	\]
    	injective when  $\catG$ is (a) a point stabilizer group scheme of $\bfG$? (b)   a parahoric
    	group scheme of $\bfG$?
    \end{que}

    The reason for introducing the question for parahoric group schemes is because the Grothendieck--Serre
    conjecture was posed only for connected groups, an assumption which is necessary in general.

    With some additional work, one can use Theorem~\ref{TH:main-coh} to show that the answer
    to both parts of Question~\ref{QU:GS} is ``yes'' when $\bfG=\uU(B,\tau)^0$ as above. This will be published elsewhere.

    We finally note that Bruhat and Tits already established a special case of part (a) in \cite[Lm.~3.9]{BruhatTits87III}:
    Assuming $R$ is a complete DVR, they show that for \emph{certain} points $y\in\calB(\bfG,F)$, the base change map
    $\rmH^1_{\et}(R,\catG_y)_{\mathrm{an}}\to \rmH^1(R,\bfG)$ is injective
    (the group scheme $\catG_y$ is denoted $\mathbf{N}_H(P)$ in \cite{BruhatTits87III} where $P=\catG_y^0(R^{\sh})$
    and $H=\bfG(F^{\sh})$, cf.\
     \cite[\S1.7, \S3.5]{BruhatTits87III}). Here,
    the subscript ``$\mathrm{an}$'' denotes the subset of cohomology classes $\alpha$ for which the closed fiber
    of the $\alpha$-twist ${}^\alpha \catG_y\to \Spec R$ has no proper parabolic subgroups; see \cite[\S3.6]{BruhatTits87III}.
    The points $y$ for which this result applies are those points with the property that $G_y\supseteq G_z$ for any $z$ in the same
    facet as $y$. For example, when $\bfG$ is semisimple, this holds for the center of mass of any facet.

\section{Hermitian Forms Equipped with  a Group Action}
\label{section:Gamma-forms}

    In this  section, we apply Theorem~\ref{TH:non-unimodular-forms} to prove a
    result about
    hermitian forms equipped with an  action of a finite group. Throughout, let $R$ denote a semilocal PID
    with $2\in\units{R}$, let $F$ be the fraction field of $R$, let $u\in\{\pm 1\}$, and let $\Gamma$ be a finite group.
    We let $R\Gamma$ denote the group ring of $\Gamma$ over $R$.

\medskip

    Recall that a \emph{$u$-hermitian $\Gamma$-form}, or just \emph{$\Gamma$-form}, consists of a pair $(P,f)$
    such that $P$ is a right $R\Gamma$-module, $f:P\times P\to R$
    is a $u$-hermitian form over $(R,\id_R)$ (so  $P\in\rproj{R}$),
    and
    $f(xg,yg)=f(x,y)$ for all $x,y\in P$ and $g\in \Gamma$.
    An \emph{isomorphism of $\Gamma$-forms} from $(P,f)$ to another $\Gamma$-form
    $(P',f')$ is an isomorphism of $R\Gamma$-modules
    $\phi:P\to P'$ such that $f'(\phi x,\phi y)= f(x,y)$ for all $x,y\in P$.
    Scalar extension of $\Gamma$-forms is defined in the obvious way.
    For an extensive discussion about $\Gamma$-forms, see \cite{Riehm11}.

\medskip

    Note that if $P$ is a right $R\Gamma$-module,
    then $P^*:=\Hom_R(P,R)$ admits a \emph{right} $R\Gamma$-module
    structure given by linearly extending $(\phi g)x=\phi(x g^{-1})$
    ($\phi\in P^*$, $g\in \Gamma$, $x\in P$).
    It is easy to check that
    a $u$-hermitian form $f:P\times P\to R$ is a $\Gamma$-form if and only
    if $f_\ell:P\to P^*$ is a homomorphism of $R\Gamma$-modules.
    In this case, the coradical $\corad(f)=\coker(f_\ell)$
    is a right $R\Gamma$-module.

    We say that a $\Gamma$-form is nearly unimodular if it is nearly unimodular
    as a $u$-hermitian form over $R$.

    \begin{example}
        Let $K/F$ be a finite field extension and let $\Gamma\to \Gal(K/F)$
        be a group homomorphism. Then $\Gamma$ acts on $K$.
        Let $S$ be the integral closure of $R$ in $K$. Then the trace form
        $(x,y)\mapsto \mathrm{tr}_{K/F}(xy):S\times S\to R$
        is a $\Gamma$-form.
    \end{example}

    \begin{thm}\label{TH:Gamma-forms}
        Let $(P,f)$, $(P',f')$ be two nearly unimodular $\Gamma$-forms over $R$
        whose coradicals are isomorphic as $R\Gamma$-modules.
        Assume that $|\Gamma|\in\units{R}$.
        Then $(P_F,f_F)\cong (P'_F,f'_F)$ as $\Gamma$-forms implies
        $(P,f)\cong (P',f')$ as $\Gamma$-forms. Furthermore, any unimodular $\Gamma$-form over $F$
        is obtained by base change from a nearly unimodular $\Gamma$-form over $R$.
    \end{thm}

    We set  notation for the proof:
    Let $A=R\Gamma$. The ring $A$ has an $R$-involution
    $\sigma:A\to A$ given by $(\sum_{g\in\Gamma}a_gg)^\sigma=\sum_{g\in\Gamma}a_gg^{-1}$.
    Let $P$ be a right $A$-module.
    To avoid ambiguity, we let $P^\circ$ denote $\Hom_A(P,A)$ (viewed
    as a right $A$-module as in~\ref{subsection:hermitian-forms}), while
    $P^*$  denotes $\Hom_R(P,R)$ (also viewed as a right $A$-module).
    Finally,
    let $\calT:A\to R$ be given by
    \[
    \calT\big(\sum_{g\in \Gamma}a_gg\big)=a_{1_\Gamma}\ .
    \]
    Theorem~\ref{TH:Gamma-forms} now follows from the following proposition,
    which reduces everything to the setting of Theorems~\ref{TH:non-unimodular-forms} and~\ref{TH:maximal-orders}.

    \begin{prp}
    	Assume $|\Gamma|\in\units{R}$. Then:
    	\begin{enumerate}
    		\item[{\rm(i)}] $A$ is separable  over $R$ (and hence a maximal $R$-order by Theorem~\ref{TH:separable-maximal-hereditary}).
    		\item[{\rm(ii)}] There is an isomorphism between $\herm[u]{A,\sigma}$, 
    		the category of all $u$-hermitian spaces over $(A,\sigma)$ (cf.\ \ref{subsection:hermitian-forms}),
    		and the category of $\Gamma$-forms given by $(P,f)\mapsto (P,\calT\circ f)$;
    		isometries are mapped to themselves.
    		\item[{\rm(iii)}] The isomorphism in (ii) is compatible with base change
    		and it preserves coradicals.
    		\item[{\rm(iv)}] A right $A$-module $M$ is semisimple if and only if
    		it is semisimple as an $R$-module.
    	\end{enumerate}
    \end{prp}

    Notice that part (iv)
    implies that a $\Gamma$-form $(P,f)$ is nearly unimodular if and only if its
    coradical is semisimple as an $R\Gamma$-module.


    \begin{proof}
    	(i) See for instance \cite[p.~41]{DeMeyIngr71SeparableAlgebras}.
    	
    	(ii) Observe first that any $A$-module
    	which is f.g.\ projective over $R$ is projective as an $A$-module by Proposition~\ref{PR:lattices}
    	(see \cite[Pr.~2.14]{Sa99} for a more direct proof). Using this, we construct an
    	inverse to $(P,f)\mapsto (P,\calT\circ f)$ as follows:
        For every $\Gamma$-form $(P,h)$, define $\hat{h}:P\times P\to A$
        by $\hat{h}(x,y)=\sum_{g\in\Gamma}h(xg,y)g$.
        It is routine to check that $(P,h)\mapsto (P,\hat{h})$  defines an inverse of $(P,f)\mapsto (P,\calT\circ f)$.

        (iii) The compatibility with scalar extension is  straightforward.

        Observe that the functors
        $*$ and $\circ$ from $\rMod{A}$ to $\rMod{A}$ are naturally
        isomorphic. Indeed, for all $P\in\rMod{A}$,
        define $\Phi_P:P^\circ\to P^*$ by $\Phi_P \phi=\calT\circ \phi$
        and $\Psi_P:P^*\to P^\circ$ by
        $(\Psi_P\psi)x=\sum_{g\in\Gamma}\psi(x g)g^{-1}$.
        It is easy to check that $\Phi=\{\Phi_P\}_{P\in\rMod{A}}:\circ \to *$
        and  $\Psi=\{\Psi_P\}_{P\in\rMod{A}}:*\to \circ$ are well-defined natural transformations
        which are inverse to each other, hence our claim.
        Now, if $(P,f)$ is a $u$-hermitian form over $(A,\sigma)$ and $h=\calT\circ f$,
        then it is easy to check that $\Phi_P\circ f_\ell = h_\ell$. Thus, since
        $P^*\cong P^\circ$ and $\Phi_P$ is an isomorphism, $\corad(f)=\coker(f_\ell)\cong\coker(h_\ell)=\corad(h)$,
        so the isomorphism in (ii) preserves coradicals.

        (iv) Write $k=R/\Jac(R)$. Then $A_k\cong A/A\Jac(R)$ is separable over $k$,
        which is a finite product of fields, and hence $A_k$ is semisimple (see~\ref{subsection:orders}).
        On the other hand $A\Jac(R)\subseteq\Jac(A)$ by 
        Proposition~\ref{PR:finite-alg-over-semilocal}, so $\Jac(A)=A\Jac(R)$.
        It follows that if $M$ is semisimple  as an $R$-module or as an $A$-module, then we may view it as
        a module over  $A_k=A/\Jac(A)$, and in particular over $k=R/\Jac(R)$. Since both $A/\Jac(A)$ and $R/\Jac(R)$
        are semisimple, $M$ must be semisimple both as an $A$-module and as an $R$-module.
%
    \end{proof}

    \begin{remark}
        The equivalence of the functors $*$ and $\circ$ in part (ii)
        holds even when $|\Gamma|\notin\units{R}$. More generally, it
        holds when $A$ is a  \emph{symmetric $R$-algebra};
        see {\cite[\S16F, Th.~16.71]{La99}} for further details.
        The equivalence between the categories of hermitian forms and $\Gamma$-forms also holds without assuming $|\Gamma|\in \units{R}$,
        provided one allows hermitian forms to have arbitrary base modules.
    \end{remark}

    \begin{remark}
        We do not know if the assumption $|\Gamma|\in\units{R}$
        in Theorem~\ref{TH:Gamma-forms} is necessary.
        However,
        by \cite{Dicks79}, $R\Gamma$ is not hereditary
        when $|\Gamma|\notin \units{R}$, so one cannot
        treat this case
        using Theorem~\ref{TH:non-unimodular-forms} and its consequences.
    \end{remark}

\bibliographystyle{plain}
\bibliography{MyBib_16_05}

\end{document}